\documentclass[11pt]{article}
\usepackage{fullpage}
\usepackage{fancyhdr}
\usepackage{epsfig}
\usepackage{amsmath,amssymb,amsthm}
\usepackage{amsmath,amscd}

\newcommand{\Z}{\mathbb{Z}}

\newcommand{\C}{\mathbb{C}}
\newcommand{\R}{\mathbb{R}}
\newcommand{\Q}{\mathbb{Q}}

\newcommand{\Ok}{\mathcal{O}_K}

\newcommand{\G}{Gal}
\newcommand{\gotha}{\mathfrak a}
\newcommand{\gothp}{\mathfrak p}

\newcommand{\gothm}{\mathfrak m}
\newcommand{\gothn}{\mathfrak n}

\newcommand{\Hom}{\text{Hom}}
\newcommand{\Spec}{\text{Spec}}

\newcommand{\Sp}{\text{Sp}}

\newcommand{\rig}[1]{#1^{rig}}
\newcommand{\ang}[1]{\langle #1 \rangle}

\renewcommand {\bar}{\overline}
\newcommand{\brac}[1]{\big[ #1 \big] }

\newtheorem{theorem}{Theorem}[section]
\newtheorem{lemma}[theorem]{Lemma}

\newtheorem{definition}{Definition}

\newtheorem{proposition}{Proposition}

\newtheorem{corollary}[theorem]{Corollary}

\newtheorem*{remark}{Remark}

\def\frak{\relaxnext@\ifmmode\let\next\frak@\else
 \def\next{\Err@{Use \string\frak\space only in math mode}}\fi\next}
\def\goth{\relaxnext@\ifmmode\let\next\frak@\else
 \def\next{\Err@{Use \string\goth\space only in math mode}}\fi\next}
\def\frak@#1{{\frak@@{#1}}}
\def\frak@@#1{\noaccents@\fam\euffam#1}
\font\tengoth=eufm10
\newfam\gothfam \def\goth{\fam\gothfam\tengoth} \textfont\gothfam=\tengoth

\begin{document}

\title{$p$-adic $L$-functions on Hida Families}
\author{Joe Kramer-Miller}
\date{}
\maketitle

\section{Overview}

A major theme in the theory of $p$-adic deformations of automorphic 
forms is how $p$-adic $L$-functions over eigenvarieties relate to 
the geometry of these eigenvarieties.  In this article we prove
results in this vein for the ordinary part of the eigencurve 
(i.e. Hida families).  We address how Taylor expansions of 
one variable $p$-adic $L$-functions varying over families 
can detect "bad" geometric phenomena: crossing components 
of a certain intersection multiplicity and ramification 
over the weight space.  Our methods involve proving a 
converse to a result of Vatsal relating congruences between 
eigenforms to their algebraic special $L$-values and then 
$p$-adically interpolating congruences using formal models.  

\subsection{Congruences between Cusp Forms and Special Values}
The connection between algebraic parts of special values and
congruences between eigenforms was observed by Mazur.  
The underlying principle is that 
congruent forms should have congruent special values.  
General results were proven by Vatsal (\cite{Vat}) in 
a number of different situations in this direction with applications
towards nonvanishing theorems in mind.  
In particular, let $N>3$ and $k>1$.  Let $p>3$
be a prime and let
$K$ be an extension of $\Q_p$ containing the Fourier 
coefficients of all normalized eigenforms for the congruence
subgroup $\Gamma_1(N)$ and weight $k$.  Let 
$\mathbb{T}_{N,k}$ be the Hecke algebra of 
$S_k(N, \mathcal{O}_K)$.  A maximal ideal $\gothm$ of 
$\mathbb{T}_{N,k}$ corresponds to a residual 
Galois representation $\overline{\rho}$.  We will
make the assumption that 
\[H_1(\mathbb{H}/\Gamma_1(N), \mathcal{L}_n(\mathcal{O}_K))_\gothm^\pm 
\cong \mathbb{T}_{N,k, \gothm} \]
as $\mathbb{T}_{N,k}$-modules, where $\mathcal{L}_n(\mathcal{O}_K))$ is the
local system associated to $L_n(\mathcal{O}_K)$.  
This isomorphism is unique up to
an element in $\mathcal{O}_K^*$ and a choice of isomorphism
corresponds to choosing periods.  If $f$ and $g$ are two eigenforms
with residual representation $\overline{\rho}$, then we have two
$\mathcal{O}_K$-algebra homomorphisms, $\delta_f$ and $\delta_g$,
from $\mathbb{T}_{N,k, \gothm}$ to $\mathcal{O}_K$.  Any congruence
satisfied between $f$ and $g$ is necessarily satisfied between 
$\delta_f$ and $\delta_g$.  Evaluating $\delta_*$ on
the appropriate cycle (maybe extending scalars to include the 
necessary roots of unity) yields special $L$-values.  These
special values must satisfy any congruence between $f$ and $g$.  \\

In this article we prove that the converse is true.  We show that
if periods $\Omega_f$ and $\Omega_g$ can be chosen so that the
algebraic special values of both eigenforms are congruent mod $p^r$, then
we have $f\cong g \mod p^r$.  To prove this result, we use the
theory of modular symbols introduced by Manin \cite{Manin} and
generalized further by Ash and Stevens \cite{Ash-Stevens}.  We show
that a modular symbol is completely determined by its special values
(in fact, finitely many special values are needed) and then use
a standard congruence module argument (see \cite{Hida5} or \cite{Ribet})
with a Hecke module made up of modular symbols.  This result can be 
reinterpreted using the $p$-adic $L$-function constructed in
\cite{MTT}.  By the uniqueness of the $p$-adic $L$-function, we
conclude that we only need to consider the special values
$L^{\text{alg}}(f,\chi, 1)$ as opposed to all critical values 
between $1$ and $k$.  Combining the results from the first
half of the article with Vatsal's result gives:

\begin{theorem} 
Let $f$ and $g$ be eigenforms as above and assume that 
\[H_1(\mathbb{H}/\Gamma_1(N), \mathcal{L}_n(\mathcal{O}_K))_\gothm^\pm 
\cong \mathbb{T}_{N,k, \gothm}.\]

Then the following are equivalent:

\begin{itemize}
\item The forms $f$ and $g$ are congruent modulo $p^r$
\item There exists periods $\Omega_f^\pm$ and $\Omega_g^\pm$
such that for all Dirichlet characters $\chi$ we have
\[ \tau(\chi)\frac{L(f,\chi,1)}{2\pi i \Omega_f^\pm} \cong 
\tau(\chi) \frac{L(g,\chi,1)}{2\pi i \Omega_g^\pm} \mod p^r, \] where
$\tau(\chi))$ denotes the Gauss sum.
\item There exists $N>0$ and periods $\Omega_f^\pm$ and $\Omega_g^\pm$
such that for all Dirichlet characters $\chi$ of character less than
$N$ we have
\[ \tau(\chi)\frac{L(f,\chi,1)}{2\pi i \Omega_f^\pm} \cong 
\tau(\chi) \frac{L(g,\chi,1)}{2\pi i \Omega_g^\pm} \mod p^r, \] where
$\tau(\chi))$ denotes the Gauss sum.

\item There exists $p$-adic $L$-functions defined using the same periods
such that $L_p(f, \chi, s) - L_p(g, \chi, s)$ is divisible by $p^r$ 
(here $L_p(f,\chi, s)$ denotes the cyclotomic $p$-adic $L$-function twisted
by the Dirichelt character $\chi$) for all $\chi$.

\end{itemize}
\end{theorem}

\subsection{Crossing components in Hida families}
In the second part of this article we prove a geometric analogue to
the results of the first part.  In the first part we are concerned 
with congruences between cusp forms of powers of $p$ of level $N$.  
This corresponds
to a geometric picture in $X:=\Spec{\mathbb{T} \otimes \mathcal{O}_K}$.  
The points of co-dimension zero in $X$ correspond to cuspidal eigenforms
of level $N$.  The points of co-dimension one correspond to 
residual representation.  Let $x_f$ and $x_g$ be co-dimension zero
points of $X$ corresponding to eigenforms $f$ and $g$.  Then
$x_f$ and $x_g$ specialize to the same co-dimension one point $x_\rho$ if and
only if $f \cong g \mod \pi_K$.  We define the 
"intersection multiplicity"
of the components $\overline{x_f}$ and $\overline{x_g}$ at $x_\rho$ to be
\[\dim_{\mathcal{O}_K/\phi_K} \mathbb{T}_{x_\rho} /(\gothp_f + \gothp_g)\]
where $\gothp_*$ is the prime corresponding to $x_*$.  This definition
agrees with the algebraic definition provided in \cite{Fulton2}.  The
largest power of $\pi_K$ for which $f$ and $g$ are congruent is
equal to this intersection multiplicity.  The results from the first part
may be reformulated to relate congruences between special $L$-values and 
the intersection multiplicity of 
$\overline{x_f}$ and $\overline{x_g}$.

This geometric interpretation of congruences
between eigenforms suggests analogues for Hida families.  Congruences
between connected components of Hida families corresponds to crossing
components in the generic fiber.  Instead of looking at congruences for
different powers of $\pi_K$ we will be interested in the 
intersection multiplicity of two crossing components.  Any two such
components correspond to minimal primes of $\mathbb{T}_\gothm$ 
where $\gothm$ is a maximal
prime of the big Hecke-Hida algebra. \\

The $p$-adic $L$-functions we will be interested can be described as follows.
For a fixed Dirichlet character $\chi$ we will describe $L_p(\chi, s) \in \mathbb{T}_\gothm$
that interpolates $L(f, \chi, 1)^{alg}$ as $f$ varies over the hida family.  
This $L$-function is obtained by fixing the "twisted" cyclotomic variable
in a slightly modified version of the $p$-adic $L$-function described in \cite{EWP}.
For a
connected component $C$ of $\Spec(\mathbb{T}_\gothm)$ we may restrict $L_p(\chi, s)$ to $C$,
which we call $L_p(C, \chi, s)$.  If the natural projection onto the weight space
$\pi: C \to \Spec\mathcal{O}_K[[T]]$ is etale
at a point $x$, then we may find a small enough affinoid neighborhood $U$ around $x$
so that $L_p(C,\chi, s)|_U$ may be written as a power series $f_{x, \chi, C}(T)$ in a canonical
way.  We may now state
our main result:

\begin{theorem} \label{Crossing Theorem} Let $C_1$ and $C_2$ be two components of $\rig{\Spec(\mathbb{T}_\gothm)}$
crossing at $x$.  Assume that $\pi$ is etale at $x$ and that the weight of $x$ is a
limit of classical weights (e.g. any $\mathcal{O}_K$-valued weight.)  Let $I$ be the 
intersection multiplicity of $C_1$ and $C_2$ at $x$.  Then $f_{x, \chi, C_1}(T)$
and $f_{x,\chi, C_2}(T)$ are congruent modulo $(T-\pi(x))^{I}$.  In other words,
the Taylor expansions of $f_{x, \chi, C_1}$ and $f_{x,\chi, C_2}$ agree for the
first $I$ terms.  For some character $\chi$', the two $L$-functions differ
at the $I+1$-th term.  
\end{theorem}

To prove this theorem, we reduce the problem to the situation where
both components look \emph{almost} like $\Spec(\mathcal{O}_K[[T]])$.  This involves
choosing small affinoid neighborhoods of $x$ in the rigid fiber and
choosing an appropriate formal model in the sense of Raynaud \cite{Ray}.  
This formal model is chosen in a way that allows us to remember 
congruences.  When both components look like $\Spec(\Z_p[[T]])$,
we repeatedly apply the $p$-adic Weierstrass preparation theorem to
further simplify the situation.  Finally, we will apply Theorem
\ref{p-adic cong} to a limit of classical weights approaching
$\pi(x)$. \\

It would be interesting to extend these results to the
positive slope part of the eigencurve.  There is one
technical difficulties that immediately come to mind.  
The construction of Coleman and Mazur \cite{Eigencurve}
does not
come with a formal integral model.  The large
Hecke-Hida algebra over the integers of a local field
allows us to see congruences.  Without an integral
model that captures all congruences, our methods fail.

\subsection{Ramification over the weight space}
In the final section we describe the behavior of the
$p$-adic $L$-functions at points on Hida families that
are ramified above the weight space.  Informally our result
says that a component is etale over the weight space
if and only if no poles are introduced when we differentiate
each $L$-function along the weight space.  More precisely, let $C$
be a regular connected component of a Hida family and
let $T'$ be any parameter for our weight space $\Spec(\Lambda)$.  
Our parameter defines
a derivation on the function field of $\Spec(\Lambda)$
denoted $\frac{d}{dT'}$.  This derivation extends to the function
field $K$ of $C$.  If $C$ is etale over $\Spec(\Lambda)$ then
$\frac{d}{dT'}$ will give a derivation on the global functions $A$
of $C$.  If for some $x\in C$ there exists $f\in A$ such that
$\frac{df}{dT'}$ has a pole at $x$, then $x$ must be ramified over
the weight space.  The largest pole occurring will be one
less than the ramification index.  Our main result is that
it is enough to check if there exists a Dirichlet
character such that$\frac{dL_p(C,\chi)}{dT'}$ has poles.

\begin{theorem} \label{ramification theorem} 
A regular $\mathcal{O}_K$-point $x \in C^{rig}$ is ramified over 
$\pi(x) \in \Spec(\Lambda_K)$
if and only if there exists a Dirichlet twist such
that $\frac{d}{dT} L(C,\chi)$ has a pole at $x$, where $T$ is a
parameter of the weight space.  The ramification
index of $x$ over $\pi(x)$ is equal to one more than the largest order pole occurring.
\end{theorem}

The proof of this theorem is similar to the proof of Theorem 
\ref{Crossing Theorem}.  We first take a small affinoid
neighborhood around $x$ which comes naturally equip with a formal
model that is isomorphic to $\Spec(\mathcal{O}_K\ang{Y})$.  This
allows us to apply the $p$-adic Weierstrauss preparation theorem
and Theorem \ref{p-adic cong}.

\subsection{Further Remarks}
The results in this article should have generalizations to ordinary
families of automorphic forms for larger algebraic groups.  A
several variable $p$-adic $L$-function was constructed by
Dimitrov in \cite{Dimitrov} that varies over ordinary
families of Hilbert modular forms.  It seems likely
that our geometric methods could be adapted to this context.  
Even more generally, it seems plausible that one could construct
measures using compatible families of automorphic cycles
living in Emerton's completed cohomology (see \cite{Emerton})
that detect ramification over the weight space and crossing.  \\

The author is currently investigating the extension of
the results in this paper to points of characteristic $p$.
Following the philosophy of \cite{spectral-halo} we may view these
points as the boundary of our Hida families.  These points
can be regarded as the ordinary part of the spectral halo
conjectured by Coleman.  In \cite{spectral-halo} a formal model 
is constructed
for the part of the eigencurve living over the outer Halo of the 
weight space.  It is plausible that $p$-adic $L$-functions can be
constructed on the spectral halo and that this formal model could be
used to imitate the techniques used in this paper.  \\

I would like to acknowledge my doctoral adviser
Krzysztof Klosin for introducing me to the world of automorphic
forms and encouraging me to pursue a mathematical project I am
passionate about.  I have benefited greatly from conversations
with Eric Urban, Ray Hoobler, Ken Kramer, and Johann de Jong.

\section{Modular Symbols and the Eichler-Shimura Isomorphism}

\subsection{Modular Symbols and Various Cohomology Theories}

Throughout this artical we will let $N>3$ and $\Gamma = \Gamma_1(N)$.  In
particular, $\Gamma$ is free of torsion.  
Let $D_0$ be the divisors  of $\mathbb{P}(\mathbb{Q})$ of degree $0$ with
a natural left action of $\Gamma$.  For any $Z[\Gamma]$-module $E$, we let 
$\Phi(E)=\text{Hom}_\Gamma(D_0,E)$.  These are modular symbols with 
values in $E$.  If $E$ also admits a right action of $GL_2(\Q)$ (resp. $GL_2(\Z)$) we
may define a right on $\Phi(E)$ of $GL_2(\Q)$ (resp. $GL_2(\Z)$).  If $\alpha\in \Phi(E)$
and $g\in\ GL(\Q)$ then $\alpha|_g$ sends 
$({r_1}-{r_2})$ to $\alpha({g(r_1)} - {g(r_2)})|g$.

 It is known that 
$\Phi(E)\cong H^1_c(\mathbb{H}/\Gamma, \widetilde{E})$ (see \cite[Proposition 4.1]{Ash-Stevens},) where $\widetilde{E}$ is 
the local system associated to $E$.  
The group $H^1_!(\mathbb{H}/\Gamma, \tilde{E})$ is defined to be the image of
$H^1_c(\mathbb{H}/\Gamma, \widetilde{E})$ in 
$H^1(\mathbb{H}/\Gamma, \widetilde{E})$.  We may compare topological
cohomology, which we may interpret as singular or deRham using deRham's
theorem, and group cohomology to get the following commutative diagram
whose vertical arrows are isomorphisms.\\

$\begin{CD}
\Phi(E)\cong H^1_c(\mathbb{H}/\Gamma, \widetilde{E})
 @>>>H^1_!(\mathbb{H}/\Gamma, \widetilde{E})
@>>>H^1(\mathbb{H}/\Gamma, \widetilde{E})\\
@VVV @VVV @VVV\\
 @>>> H^1_!(\Gamma,E)
@>>> H^1(\Gamma,E).
\end{CD}$
\\
\\

We call $\phi$ the map that takes a modular symbol to a $1$-cocycle.  Explicitely 
this map takes the modular symbol $\alpha$ to the $1$-cocycle that sends
$g$ to $\alpha(g(x)) - \alpha(x)$ for some $x\in \mathbb{P}(\Q)$.
The right two vertical maps send a $1$-form to $\omega$ to the $1$-cocycle
$$g \to \int_{z_0}^{g(z_0)}\omega,$$
where $z_0$ can be any number in $\mathbb{H}$.  If $\omega$ is compact
we may even allow $z_0$ to be in $x\in \mathbb{P}(\Q)$.
When $E$ is not a $\R$-vector space we must use singular cohomology, but the
idea is essentially the same.
A more thorough explanation can be found in the appendix of \cite{Hida3}.\\

\subsection{The Complex Conjugation Involution}

Each space in the above diagram has an action induced by the involution $\sigma$
 of $\mathbb{H}$ given by $z\to -\bar{z}$.  
Consider the 1-cocycle $\beta$
defined by a $1$-form $\omega_\beta$.  Then we get a new $1$-cocycle
that sends $g$ to 
$$\int_i^{-\bar{g(i)}}\omega_\beta=\int_i^{g(i)}\sigma^* (\omega_\beta).$$
Thus $\beta$ is sent to the 1-cocycle 
$g\to \beta(\xi g \xi^{-1})$, where
$\xi=\begin{pmatrix} -1 & 0 \\ 0 & 1 \end{pmatrix}$.  On 
deRham cohomology the form $\omega$ is send to its pullback
$\sigma^* (\omega)$ under $\sigma$.  In particular,
holomorphic forms are sent to anti-holomorphic forms and
vice versa.  The action on modular
symbols sends $\alpha \in \Phi(L_n(\C))$ to $\alpha|_\xi$.
It is readily checked that all of these actions are compatible
with the maps in the above diagram.  It is also clear that
we have an action on cohomology (resp. modular symbols)
with coefficients in $L_n(A)$ where $A$ is any ring in $\C$.
\\

If $V$ is any one of the spaces in the diagram above, we
have a decomposition $V=V^+ \bigoplus V^-$.  Here $V^+$ are
the elements of $V$ fixed by the action of complex conjugation
and $V^-$ are the elements negated by this action.  In fact
if $2$ is invertable in a subring $A$ of $\C$ then we 
get a diagram whose vertical arrows are isomorphisms. \\

$\begin{CD}
\Phi(L_n(A))^\pm \cong 
H^1_c(\mathbb{H}/\Gamma, \widetilde{L_n(A)})^\pm
 @>>>H^1_!(\mathbb{H}/\Gamma, \widetilde{L_n(A)})^\pm
@>>>H^1(\mathbb{H}/\Gamma, \widetilde{L_n(A)})^\pm \\
@VVV @VVV @VVV\\
 @>>> H^1_!(\Gamma,L_n(A))^\pm
@>>> H^1(\Gamma,L_n(A))^\pm.
\end{CD}$
\\
\\

\subsection{The Eichler Shimura Isomorphism}
Let $f\in S_{n+2}(\Gamma)$ and let $L_n$ be the space of homogeneous 
polynomials in $x$ and $y$ of degree $n$.  The $1$-form $f(z)(x-zy)^ndz$, which
takes values in $L_n$ is invariant under
$\Gamma$ and therefore gives us an element $\omega_f$ of 
$H^1(\mathbb{H}/\Gamma, \widetilde{L_n}(\C))$.
The form $\omega_f$ does not have compact support, but it turns out that
$\omega_f \in H^1_!(\mathbb{H}/\Gamma, \widetilde{L_n}(\C))$.  It turns out that
if we take the real part of $\omega_f$ we get an isomorphism
$$S_{n+2}(\Gamma) \cong H^1_!(\mathbb{H}/\Gamma, \widetilde{L_n}(\R))
\cong H^1_p. $$ \\
We easily check that the action of complex 
conjugation sends $f(z)(x-zy)^ndz$ to
$-f(-\bar{z})(x+\bar{z}y)^nd\bar{z}$.  In particular
we see that the projection of $\omega_f$ onto
$H^1(\mathbb{H}/\Gamma, \widetilde{L_n(\C)})^\pm$ is
$$\frac{f(z)(x-zy)^ndz \pm -f(-\bar{z})(x+\bar{z}y)^nd\bar{z}}
{2}.$$

\subsection{Integral Cohomology}
Let $f$ be an eigenform for $\Gamma$ and let $\omega_f^\pm$
be the corresponding $1$-form in 
$H^1(\mathbb{H}/\Gamma, \widetilde{L_n(\C)})^\pm$.  We define
a modular symbol $\alpha_f^\pm$ by
\[\alpha_f^\pm (\{r_1\} - \{r_2\}) = \int_{r_1}^{r_2} \omega_f^\pm.\] 
This 
gives a Hecke equivariant map $s: S_k(\Gamma) \to \Phi(L_n(\C))^\pm$.
By a theorem of Shimura (see \cite[Theorem 4.8]{Greenberg-Stevens}) the subspace
of $\Phi(L_n(\C))^\pm$ that has the same Hecke eigenvalues as $f$ is
one dimensional.  If $A$ is a subring of $\C$ containing the Hecke
eigenvalues or $f$, there exists periods $\Omega_f^\pm$ such that
\[\frac{\alpha_f^\pm}{\Omega_f^\pm} \in \Phi(L_n(A))^\pm. \]  The map
$s$ is a section of the map 
$\phi:\Phi(L_n(\C))^\pm \to H^1(\Gamma, \widetilde{L_n(\C)})^\pm$.
We know that $\phi(L_n(A))^\pm \subset H^1(\Gamma, \widetilde{L_n(A)})^\pm$
by our explicit description of $\phi$.

\section{Congruences Between Cusp Forms and L-functions}
In this section we prove that two cusp forms are congruent
if and only if the "algebraic" special values of their
$L$-functions admit congruences for all twists.  It turns
out that we only need to consider finitely many twisted
$L$-functions to determine if there are congruences.

\subsection{Special Values of Modular Symbols}
Let $\brac{\frac{p}{q}}$ denote the degree zero divisor 
$\{\frac{p}{q}\} - \{\infty\}$. 
For a primitive Dirichlet character $\chi$ of conductor $D$ we define
$$\Lambda(\chi)= \sum_{m=0}^{D-1} 
\bar{\chi(m)}\brac{\frac{p}{q}} \in D_0 \otimes \Z[\chi].$$
Let $E$ is a $\Z_p[\Gamma]$-module and let $\alpha\in E$.
Define the special value
 $L(\alpha,\chi)$ of $\alpha$ to be $\alpha(\Lambda(\chi))$.  If $\alpha$
 is the modular symbol associated to a cusp form then the special
 values of $\alpha$ are related to the special values of the forms
 $L$-functions.  The next proposition says that a modular symbol is 
 completely determined by special values and that we have some
 control over which special values we need to look at  
 More specifically, let $\epsilon>0$ and
 let $r \in \Z$ be prime to $p$.  Take $X$ to be the set of primes $q$ larger 
 than $\epsilon$ that satisfy the congruences
 \begin{eqnarray*}
 q & \cong & r \mod p \\
 q & \cong & 1 \mod N^2.
 \end{eqnarray*}
 We will prove that $\alpha$ is determined by its special values 
 for Dirichlet characters with conductor in $X$.  This
 type of result was first observed by Stevens (Lemma 2.1 in \cite{Stevens})
 for weight $2$ forms.

\begin{lemma} Let $\frac{c}{d}$ be a reduced fraction whose denominator
is $1 \mod N$.
There exists $\gamma \in \Gamma$ such that the denominators of 
$\gamma(\frac{c}{d})$ and $\gamma(0)$
are in $X$.  
\end{lemma} \label{elementary lemma}

\begin{proof}

First assume that $p\not | c$.  If this is not the case, 
replace $\frac{c}{d}$ with 
\[ \begin{pmatrix} 1 & 1 \\ 0 & 1 \end{pmatrix} \frac{c}{d} = \frac{c+d}{d}, \]
whose denominator is prime to $p$.  A similar manipulation will guarantee $c$ is
coprime to $N$ as well.  
Let $l_1$ be a prime number that is congruent to $1$ modulo $N^2$ and congruent
to $r$ modulo $p$.  We may take $l_1$ large enough to be contained in $X$
and so that $l_1\not | c$.
Next take $z$ to be a prime satisfying
\begin{eqnarray*}  
z & \equiv & 1 \mod N^2, \\
z & \equiv & dl_1 \mod Nc, \text{ and} \\
 z &\equiv & r \mod p.
 \end{eqnarray*}
Then $z$ can be written as $yNc + dl_1$.  We set $l_2=Ny$. 
There is a matrix $\gamma=\begin{pmatrix} t_1 & t_2 \\ l_1 & l_2 \end{pmatrix}$
in $\Gamma$ and we may assume that $t_1$ is divisible by $z$.  We
see that $\gamma$ satisfies the desired properties.

 \end{proof}

 \begin{lemma} \label{special lemma} 
 Let $\alpha$ be a module symbol with values in $E$ and assume
 that $\alpha$ doesn't map to zero in $H^1(\Gamma, E)$.
 Let $r$ be prime to $p$ with
 $r \not \equiv 1 \mod p$ and let
 $\epsilon > 0$.  
Then either $L(\alpha, \chi_{triv})\neq 0$ or there exists
 a primitive Dirichlet character 
$\chi$ whose conductor is in $X$ such that $L(\alpha,\chi)\neq 0$.

\end{lemma}
\begin{proof}

First, extend any modular symbol in $\Phi(E)$ to an element of
$\Hom_\Gamma (D_0 \otimes_{\Z} \overline{\Z_p},E\otimes_{\mathcal{O}}\overline{\Z_p})$.
For any subset $S$ of $\Q$, we will define 
\[ A(S) := \bigoplus_{x \in S} \overline{\Z_p}\brac{x}, \] 
\[ A'(S) := \Hom(A(S), E\otimes_{\mathcal{O}}\overline{\Z_p}), \]
\[ \rho_S:= \Hom_\Gamma (D_0 \otimes_{\Z} \overline{\Z_p},E\otimes_{\mathcal{O}}\overline{\Z_p})
\to A'(S), \]
where $\rho_S$ is the natural map. Note that
the map $\rho_{\Q}$ is an injection and that if $S\subset T \subset \Q$ then
$\rho_S$ factors through $\rho_T$.  Let $S_X$ be the set of rational numbers
whose denominator is in $X$.  Let $S_X' = S_X \cap (0,1]$.  We will
be interested in $\rho_{S_X}$ and $\rho_{S_X'}$.  \\

Let $\alpha$ be a modular symbol with $L(\alpha, \chi)=0$ for all 
 $\chi$ whose conductor is in $X$.  For $m\in X$, we let $S_m$
 be the rational numbers $\frac{1}{m},...,{m-1}{m}$.  The 
 $\overline{\Z_p}$-span of $\Lambda(\chi)$ for all
 $\chi$ of conductor $m$ is a submodule $M$ of $A(S_m)$.  
 Then $M$ consists of  all elements $\Sigma a_i\frac{i}{m}$ where
 $\Sigma a_i = 0$.  That is to say, $M$ is the kernel of
 the augmentation map on $A(S_m)$.  Consider $\brac{\frac{N}{m}} \in A(S_m)$.
 We have
 \begin{eqnarray*}
 \alpha(\brac{\frac{N}{m}}) & = & \alpha(\brac{\frac{N}{m}}) - \alpha(\brac{0}) \\
 &  = & \alpha(\{\frac{N}{m}\} - \{0\}),
 \end{eqnarray*}
where $\alpha(\brac{0})=0$ because $L(\alpha, \chi_{triv})=0$.  Let
$\gamma_0 = \begin{pmatrix} 1 & 0 \\ -Nk & 1 \end{pmatrix}$, where
$m=N^2k + 1$.  Then 
\begin{eqnarray*}
\gamma_0(\{\frac{N}{m}\} - \{0\}) &=& \{-N\} - \{0\}\\
&=& \brac{-N} - \brac{0}.
\end{eqnarray*}
There is a parabolic element $\gamma_1 \in \Gamma$ such that 
$\gamma_1(\brac{N}) = \brac{0}$, so $\alpha(\brac{-N})=0$.
It follows that $\alpha(\gamma_0(\{\frac{N}{m}\} - \{0\}))=0$
and thus $\alpha(\{\frac{N}{m}\} - \{0\})) = \alpha(\brac{\frac{N}{m}}) = 0$.
Since the span of $M$ and $\brac{\frac{N}{m}}$ is all of $A(S_m)$ we see
that $\alpha$ is zero on all of $A(S_m)$.  In other words $\rho_{S_m}=0$.

Let $x\in S_X$.  There is 
a parabolic element $\gamma' \in \Gamma$ such that $\gamma'(x) \in S_X'$.
Applying $\gamma'^{-1}$ to the equation $\alpha(\gamma'(\brac{x}))=0$
shows that $\alpha(\brac{x})=0$.  It follows that $\rho_{S_X}(\alpha)=0$.
Now we consider the divisor $\{0\} - \{\frac{c}{d}\}$,
where $\frac{c}{d} = \gamma(0)$ for some $\gamma \in \Gamma$.  The denominator
$c$ is $1 \mod N$, so by \ref{elementary lemma} there
exists $\gamma_0 \in \Gamma$ such that $\gamma_0(0)$ and $\gamma_0(\frac{c}{d})$
are in $S_X$.  Then we have

\begin{eqnarray*}
\alpha(\{\gamma_0(0)\} - \{\gamma_0(\frac{c}{d}) \}) &=& 
\alpha(\brac{\gamma_0(0)} - \brac{\gamma_0(\frac{c}{d}))}) \\
&= & 0,
\end{eqnarray*}
since $\alpha$ is in the kernel of $\rho_{S_X}$.  This leads to a
contradiction, since the cohomology class of $\alpha$ in $H^1(\Gamma, E)$
is represented by the $1$-cocycle that sends $\gamma$ to 
$\alpha(\{0\} - \{\gamma(0)\})$.

\end{proof}

\begin{corollary} 
Let $\mathcal{O}$ be as in the lemma.  
Let $\alpha_1$ and $\alpha_2$ be modular symbols that
takes values in $L_n(\mathcal{O})$ and let $I$ be an ideal of $\mathcal{O}$.  
The following are equivalent.

\begin{itemize}
\item For $r, \epsilon$ and $X$ be as in the lemma 
\[ L(\alpha_1, \chi) \equiv L(\alpha_2, \chi) \mod I \]
for all $\chi$ with conductor in $X$.

\item $\alpha_1 \equiv \alpha_2 \mod IL_n(\mathcal{O})$.
\end{itemize}

\end{corollary}

\begin{proof}

By (1.16) of \cite{Hida5} we know that 
$\Phi(L_n(\mathcal{O})) \otimes \mathcal{O}/I \cong
\Phi(L_n(\mathcal{O}/I))$.  Consider the image $\beta$
of the modular symbol $\alpha_1-\alpha_2$ in $\Phi(L_n(\mathcal{O}/I))$.
If $L(\alpha_1, \chi) \equiv L(\alpha_2, \chi) \mod I$ then 
$L(\beta, \chi) = 0$.  Then from the lemma, we see the first
condition implies the second.  The other direction is immediate.

\end{proof}

\subsection{Special Values of L-functions}
Let $f(z)\in S_k(\Gamma )$ 
and write $f(z)= \Sigma a_n q^n$ where 
$q=e^{2i \pi z}$.  Then $L(s,f)$ is defined to be $\Sigma a_n n^{-s}$.
We have $$ \int_0^{i\infty} f(z)z^m dz = \frac{m! L(m+1,f)}{(-2\pi i)^{m+1}}.$$
More generally, if $\chi$ is a Dirichlet character of conductor $D$ we
define $L(s,f,\chi)$ as $\Sigma a_n \chi(n) n^{-s}$.  This yields

$$ \sum\limits_{a=0}^{D-1} \overline{\chi(a)}
\int_{\frac{a}{D}}^{i\infty} f(z)z^m dz = 
\tau(\overline{\chi}) \binom{k-2}{m} m! \frac{L(m+1,f,\chi)}{(-2\pi i)^{m+1} }.$$

These two integrals tell us the special values of the modular symbol
$s(\omega_f)$.  In particular

$$s(\omega_f)(\Lambda(\chi)) = (..., \tau(\overline{\chi}) \binom{k-2}{m} m! \frac{L(m+1,f,\chi)}{(-2\pi i)^{m+1} }, ...).$$

\subsection{Congruences Between Special Values}\label{congruences between special}

In this section we use Theorem \ref{special lemma}
in conjunction with a standard congruence module argument
(cf \cite{Ribet} and \cite{Hida5}).
Let $f$ and $g$ be normalized eigenforms in $S_k(\Gamma)$
with $f(z)=\Sigma a_nq^n$ and $g(z) = \Sigma b_nq^n$.  Let
$K$ be a finite extension of $\Q$ containing the eigenvalues
of both eigenforms.  Let $v$ be a prime of $\mathcal{O}_K$ whose
residue characteristic is $p$.  Let 
$R=\mathfrak{O}_{K_{(v)}}$ be the localization of $\Ok$ at $v$
and let $\pi_v$ be a uniformizer of $R$.
Since $2$ is invertible in $R$ we have

$$ \Phi(L_n(R)) \cong \Phi(L_n(R))^+ \oplus \Phi(L_n(R))^-.$$

Let $M^+=(\C\alpha_f^+\oplus \C\alpha_g^+)\cap \Phi(L_n(R))+$.  
We remark that $M^+$ is invariant
under the Hecke operators and has dimension $2$.  Let 
$M_*^+ = M^+ \cap \C\omega_*^+$ and let $M^{*+}$ be the projection of 
$M^+$ onto $\C\omega_*^+$ (here $*$ can be $f$ or $g$).
The spaces $M_*^+$ and $M^{*+}$ are free $R$-modules of rank one
and $M_*^+ \subset M^{*+}$.  Picking a basis
of $M_*^+$ is equivalent to choosing a period $\Omega_*^+$.  The choice
of such a period is canonical up to a $v$-adic unit as discussed in
\cite{Vat}.  Let $\beta_*^+ = \frac{\alpha_*^+}{\Omega_*^+}$
be the normalized modular symbol.  Let us assume that we were able to choose
the periods so that 
$$L(\beta_f^+, \chi) \cong L(\beta_h^+, \chi) \mod v^m $$ for
all $\chi$ with conductor in $X$.  Then by the previous section we know that 
$\beta_f^+ - \beta_g^+ \in v^m \Phi(L_n(R))$.  In particular, 
we may find $x\in \Phi(L_n(R))$ with $\pi_v^m x = \beta_f^+ - \beta_g^+$.
As $x$ is in the space spanned by $f^+$ and $g^+$ we see that $M^+$
contains $x$.  Thus we have an injection 
\[ R/\pi_v^m R \to \frac{M^+}{M_f^+ \oplus M_g^+}, \text{ defined by} \]
\[ 1 \mod \pi_v^m R \to x \mod M_f^+ \oplus M_g^+. \] \\

We get
the following equivalence of Hecke modules:

\[\frac{M^{f+}}{M_f^+} \cong 
\frac{M^{f+}\oplus M^{g+}}{M^+} \cong
\frac{M^{g+}}{M_g^+}, \text{ and }\]
\[\frac{M^{f+}\oplus M^{g+}}{M^+} \cong 
\frac{M^+}{M_f^+ \oplus M_g^+}.\] In particular we find that 

\[\frac{M^{f+}}{M_f^+} \otimes R/\pi_v^m R \cong_{\mathbb{T}} 
\frac{M^{g+}}{M_g^+} \otimes R/\pi_v^m R, \] and both of these Hecke
modules are equal to $R/\pi_v^m R$ as an $R$-module.  The
Hecke operator $T_n$ acts on $\frac{M^{f+}}{M_f^+} \otimes R/\pi_v^m R $
(resp $\frac{M^{g+}}{M_g^+} \otimes R/\pi_v^m R $)
through scalar multiplication by $a_n \mod \pi_v^m R $ 
(resp $b_n\mod \pi_v^m R $).  The isomorphism of Hecke modules then implies
$a_b \cong b_n \mod \pi_v^m R $.

\begin{theorem} \label{vatsal converse}
 Let $f$ and $g$ be normalized eigenforms in $S_k(\Gamma)$.
If there exists periods $\Omega_f^\pm$ and $\Omega_g^\pm$ (these
are defined canonically up to $p$-adic unit) that satisfy 

$$(..., \tau(\overline{\chi}) \binom{k-2}{m} m! 
\frac{L(m+1,f,\chi)}{(-2\pi i)^{m+1} \Omega_f^\pm}, ...)
\cong 
(..., \tau(\overline{\chi}) \binom{k-2}{m} m! 
\frac{L(m+1,g,\chi)}{(-2\pi i)^{m+1} \Omega_g^\pm}, ...)
\mod \pi_v^m, $$
for all characters whose conductor is in $X$ then $f \cong g \mod \pi_v^m$.
\end{theorem}

\subsection{The one variable cyclotomic $p$-adic $L$-function}
\label{The one variable cyclotomic $p$-adic $L$-function}
We will now describe the construction of cyclotomic $p$-adic $L$-function as described in
\cite{EWP}.  The only difference in what we describe is that we allow a fixed tame level of 
the Dirichlet twists that we interpolate, while \cite{EWP} only covers the tame level $1$ case.  
This construction is more or less equivalent to the function
described in \cite{MTT}.  
Let $f$ be an $p$-ordinary eigenform of weight $k\geq 2$ and level $Np^r$.  We
define $\mathbb{T}_{N,r,k}$ to be the Hecke algebra over $\mathcal{O}_K$
generated by $T_l$ for $l \not | Np$, $U_l$ for $l | Np$, and the 
diamond operators $\ang{a}$ for $a \mod N$.  Let
$\gothm$ be the maximal ideal of $\mathbb{T}_{N,r,k}$ corresponding to the residue
of $f$ modulo $p$. 
We will assume that that 
\[ H_1(\mathbb{H}/\Gamma, \widetilde{L_{k-2}}(\mathcal{O}_K))^{ord}_\gothm \cong 
(\mathbb{T}_{N,r,k})_\gothm.\]
This will be true if the residual representation is $p$-distinguished.

There is a natural map from $(\mathbb{T}_{N,r,k})_\gothm$ to $\C$ sending each Hecke
operator to it's eigenvalue on $f$.  Call this map $\delta_f$ and note that 
the image of $\delta_f$ is in $\mathcal{O}_K[f]$.
We also have a natural map from 
$H_1(\mathbb{H}/\Gamma, \widetilde{L_{k-2}}(\mathcal{O}_K))^{ord}_\gothm$ to $\C$ that is induced
by integrating each cycle along the differential $\omega_f$.  These maps are the same up
to a complex period.  This is the same period we encountered in the previous section (up to
a $p$-adic unit.)  The choice of period will be determined by the choice of isomorphism from 
our homology group to our Hecke algebra.

Let $M>1$ be prime to $p$.  We will let $\Lambda = \Z_p [[T]]$ be the standard Iwasawa 
algebra.  Define $\Lambda_M$ as $\Lambda[\Z / pM\Z^\times ]$.  
Then $\Lambda[\Z / pM\Z^\times ]\cong \varprojlim \Z_p[\Z/Mp^s \Z ^\times]$
where $1+T$ goes to the topological generator $1 + p$ of $1 + p\Z_p$.  Recall
that for any $\Z_p$ module $A$ we may think of elements of $\Lambda_M\otimes A$
as measures on $\Z_p^\times \oplus \Z / M\Z^\times$ with values in $A$.
We define an 
element $L_{M,\gothm}$ of $\Lambda_M \otimes H_1(\mathbb{H}/\Gamma, \widetilde{L_{k-2}}(\mathcal{O}_K))$
by the measure sending the open set $(a + p^r\Z_p, a + M\Z)$ in 
$\Z_p^\times \oplus \Z / M\Z$ to $U_p^{-r}\{\frac{a}{p^r M}, \infty\}\in
H_1(\mathbb{H}/\Gamma, \widetilde{L_{k-2}}(\mathcal{O}_K))$.
By the definition of $U_p$ we see that this definition is a well defined measure
(in particular it is additive).\\

When we specialize $L_{M, \gothm}$ by $\delta_f$ we obtain an element $L_{M, \gothm, f}$ 
of $\Lambda_M \otimes \mathcal{O}_K[f]$.  Specializing at certain $\C_p$ points of 
$\Lambda_M \otimes \mathcal{O}_K[f]$ will then give us special values of $L(f,\chi, s)$,
for a primitive Dirichlet character of tame level $M$.  Giving a
 $\C_p$ point of $\Lambda_M \otimes \mathcal{O}_K[f]$ is equivalent to giving an element of 
$\Hom_{cont}(\Z_p^\times \oplus \Z / M\Z^\times, \C_p^\times)$.  We
may interpret any primitive Dirichlet character $\chi$ of tame level $M$ as an
character in $\Hom_{cont}(\Z_p^\times \oplus \Z / M\Z^\times, \C_p^\times)$.
Then evaluating $L_{M, \gothm, f}$ at this character corresponding to 
$\chi$ gives $\tau(\chi)\frac{L(f,\chi, 1)}{\Omega_f}$.  If we multiply
$\chi$ by the character $(1+p)\to (1+p)^s$, then $L_{M, \gothm, f}$ 
evaluates to $\tau(\chi)s!\frac{L(f,\chi \omega^s,s+1)}{(2\pi i)^{s+1} \Omega_f} $ 
for $0\leq s < k-1$ where $\omega$ is the Tiechmuller character.
Note that
$\Spec(\Lambda_M)$ is equal to $\phi(pM)$ copies of the open unit $p$-adic ball.
One copy for each character of $\Z / pM \Z$.  For a character $\psi$ of
conductor $pM$ we let $L_{\psi, *}$ denote the restriction of 
$L_{M, \gothm, *}$ to the unit ball corresponding to $\psi$.

\begin{theorem} \label{p-adic cong} Let $f$ and $g$ be two 
$p$-ordinary eigenforms of weight $k\geq 2$ and level $Np^r$.
Let $\mathcal{O}_K$ be the ring of integers of an extension of $\Q_p$ that contains
the coefficients of $f$ and $g$.  Let $\pi$ be a uniformizer of $\mathcal{O}_K$.  There exists
an $M'$ depending on the weight and level such that the following are equivalent

\begin{itemize}
\item The forms $f$ and $g$ are congruent modulo $\pi^t$.
\item The $p$-adic $L$-functions $L_{M, \gothm, f}$ and $L_{M, \gothm, g}$ are
congruent modulo $\pi^t$ for all $M$.
\item The $p$-adic $L$-functions $L_{\psi, f}$ and $L_{\psi, g}$ are
congruent modulo $\pi^t$ for all primitive Dirichlet characters $\psi$.
\end{itemize}
\end{theorem}
\begin{proof} That congruent forms have congruent $p$-adic $L$-functions follows 
from the above discussion.  It was originally proven by Vatsal in \cite{Vat} using
the $p$-adic Weierstrass preparation theorem.  That congruent $p$-adic $L$-functions
implies a congruence of forms is a restatement of the results from the 
previous section.  
\end{proof}

\begin{corollary} Using the notation of the Theorem \ref{p-adic cong} the forms
$f$ and $g$ are congruent modulo $\pi^t$ if and only if for every Dirichlet character $\chi$
we have
\[\tau(\chi)\frac{L(f,\chi, 1)}{\Omega_f} \equiv \tau(\chi)\frac{L(g,\chi, 1)}{\Omega_g} \mod \pi^t.\]
\end{corollary}
\begin{remark} Theorem \ref{vatsal converse} tells us that the critical values
$1,...k-1$ of the twisted $L$-functions can detect congruences.  This
corollary says that it is sufficient to look at the critical value at $s=1$.
\end{remark}
\begin{proof}
The necessity is already established.  
Conversely, let us assume that the congruence between the critical values at $s=1$
 holds for all characters.
By Theorem \ref{p-adic cong} is enough
to show that $\pi^t | L_{\psi, f} - L_{\psi, g}$, where
we view $L_{\psi, *}$ as an element of $\mathcal{O}_K[[T]]$.  By
the Weierstrass preparation theorem we may write 
\[L_{\psi, f} - L_{\psi, g} = \pi^r u(T) P(T) \]
where $u(T)$ is a unit and $P(T)$ is distinguished.  Now let $\chi$
be a primitive character of conductor $p^s$.  The $\mathbb{C}_p$-point of $\Spec(\mathcal{O}_K [[T]])$ 
corresponding to $\chi$ sends
$T$ to $1 - \zeta_{p^s}$ for some primitive $p^s$-th root of unity.  The
$p$-adic valuation of $1-\zeta_{p^s}$ is $\frac{1}{\phi(p^s)}$.  As $P(T)$ is distinguished
we know that $v_p(p(1-\zeta_{p^s})) = \deg(P(T))\frac{1}{\phi(p^s)}$ if $s$ is sufficiently
large.  Putting this together with out hypothesis gives
\begin{eqnarray}
t v_p(\pi) & \leq & v_p(\pi^r u(1-\zeta_{p^s}) P(1-\zeta_{p^s})) \\ 
& = & r v_p(\pi) + \det(P(T)) \frac{1}{\phi(p^s)},
\end{eqnarray}
for sufficiently large $s$.  Letting $s$ tend to $\infty$ shows that $t \leq r$.

\end{proof}

\section{$p$-adic $L$-functions on Hida families}
In this section we describe a $p$-adic $L$-function that varies analytically
over Hida families whose residual representation are $p$-stabilized.  We will also
interpret the results of the previous section in terms of $p$-adic $L$-functions.
Throughout the remainder of this article we set $K$ to be a finite extension of $\Q_p$
and $\mathcal{O}_K$ to be the ring of integers in $K$ with uniformizing
element $\pi_K$.  We let  
$\Lambda_K:= \Lambda \otimes \mathcal{O}_K$ be the Iwasawa algebra with coefficients
in $\mathcal{O}_K$.

\subsection{Hida Families} \label{Hida theory}
In this section we give a summary of Hida theory.  For an introduction
to the theory with tame level $1$ see $\cite{Hida3}$.  For the general
situation see \cite{Hida1}.  For the most part we adopt the notation of
\cite{EWP}.
Let $M>0$ be 
relatively prime to $p$.  For $k\geq 2$, let $S_k(Np^\infty, \mathcal{O}_K)$ be the union
of all weight $k$ cusp forms for $\Gamma_1(Np^r)$ that are defined over the
$\Z_p$-algebra $\mathcal{O}_K$.  There is a natural action of $\Z/p^r\Z^\times$ on 
$S_k(Np^r\mathcal{O}_K)$ given by the product of the Nebentypus action and the character
$\gamma\to\gamma^k$.  These actions are compatible with the inclusion of
$S_k(Np^r,\mathcal{O}_K)$ in $S_k(Np^{r+s},\mathcal{O}_K)$ for any $s>0$ so that we have an action of 
$\Z_p^\times$ on $S_k(Np^\infty, \mathcal{O}_K)$.  In particular 
$S_k(Np^\infty, \mathcal{O}_K)$ is a $\mathcal{O}_K[[Z_p^\times]]$-module and a $\Lambda_K$-module.\\

For a prime $\goth{p}$ of $\Lambda_K$ of height one we write 
$\mathcal{O} (\goth{p}):= \mathnormal{\Lambda_K}/\goth{p}$.  We say that $\goth{p}$
is classical of weight $k$ if the residue has characteristic 0 and
if the composition   
$\kappa_\mathfrak{p} := 1+p\Z_p \Lambda_K \to \mathcal{O}(\mathfrak{p}) $
coincides with
the character $\lambda \to \lambda^k$ on an open subgroup of $1+p\Z_p$.\\

Let $\mathbb{T}_N$ be the $\mathcal{O}_K[[\Z_p^\times]]$-algebra generated by the
Hecke operators and diamond operators acting on $S_k(Np^\infty,\mathcal{O}_K)^\text{ord}.$
Then $\mathbb{T}_N$ is a $\Lambda_K$-algebra.  A height one prime ideal $\goth{p}$
of $\mathbb{T}_N$ is said to be classical of weight $k$ if it lies above a
weight $k$ prime of $\Lambda$.  Just as before, we set 
$\mathcal{O} (\mathfrak{p}):= \mathbb{T}_N/ \goth{p}$ and define $\kappa_{\gothp}$
to be the character from $1+p\Z_p$ to $\mathcal{O}(\gothp)$.  The fibers of
the residues of weight $k$ primes of $\Lambda_K$ in $\mathbb{T}_N$ recover
the Hecke algebra acting on weight $k$ cusp forms of tame level $N$.  
More specifically, Hida proved the following theorem in \cite{Hida1}.

\begin{theorem} Using the above notation we have:
\begin{itemize}
\item $\mathbb{T}_N$ is free of finite rank over $\Lambda_K$.
\item Let $\gothp\in \Spec (\Lambda_K)$ be a classical height one prime of weight $k$.
Then $\mathbb{T}_N \otimes \mathcal{O}(\gothp)$ is equal to the full Hecke algebra
acting on $S_k(Np^\infty, \mathcal{O}(\gothp))^\text{ord}[\kappa_\gothp]$.

\end{itemize}
\end{theorem}

The Hecke algebra $\mathbb{T}_N$ is a semi-local ring.  The different maximal
ideals correspond to the different representations of 
$\G(\bar{\mathbb{Q}},\Q)$ into $GL_2 (\bar{\mathbb{F}_p})$ that arise as
residues of representations associated to cusp forms of tame level $N$.  
The classical height one primes of $\mathbb{T}_N$ are in one to one 
correspondence with the Galois conjugacy classes of
 eigenforms in $S_k(Np^\infty, \mathcal{O}_K)$.  The minimal primes represent
 maximal families of eigenforms.  In particular, Let $\mathfrak{a}$
 be a minimal prime of $\mathbb{T}_N$.  Consider the formal power series
 $$f(q) = \Sigma T_n q^n \text{ with } T(n) \in 
 \mathbb{T}_N / \mathfrak{a}.$$  Then the image of $f(q)$ in 
 $\mathbb{T}_N/ \mathfrak{a} \otimes \mathcal{O}(\gothp)$, where $\gothp$
 is a classical height one prime, will give the
 Fourier expansion of the modular form corresponding to $\gothp$.

\subsection{The two variable cyclotomic $p$-adic $L$-function}
Let $\mathfrak{m}$ be a maximal prime of $\mathbb{T}_N$.  Then the
classical height one primes of
$\mathbb{T}_{N,\mathfrak{m}}$ (the localization of $\mathbb{T}_N$ at the
prime $\mathfrak{m}$) correspond to eigenforms with the same residual
representation.  Equivalently, the Fourier coefficients of all
classical height one primes of  $\mathbb{T}_{N,\mathfrak{m}}$ 
are equivalent in the appropriate residue field.  From now on we will assume
that this residual
representation is $p$-distinguished and irreducible.  This is a necessary
condition to use the previous section.

\begin{theorem}
Let $\mathfrak{m}$ be a maximal prime of $\mathbb{T}_{N,r,k}$ whose
residual representation is irreducible and $p$-distinguished.  Then 
$H_1(\mathbb{H}/\Gamma, \widetilde{L_{k-2}}(\mathcal{O}_K))^{\pm ord}_\gothm \cong 
(\mathbb{T}_{N,r,k})_\gothm$ as $\mathbb{T}_{N,r,k}$-modules.
\end{theorem}
\begin{proof} \cite[Proposition 3.1.1]{EWP}
\end{proof}

We define 
$H_1(Np^\infty, \mathcal{O}_K) :=\varprojlim H_1(\mathbb{H}/\Gamma_1(Np^r), \mathcal{O}_K)^{ord}_\gothm$.  
It is proven in \cite{EWP} that 
$H_1(Np^\infty,\mathcal{O}_K)^{\pm ord}$ is a free rank $1$
$\mathbb{T}_N$-module.  We also have

\begin{theorem} Let $\gothp_{N,r,k}$ be the product of all primes of weight $k$
in $\mathbb{T}_N$.  Then 
$$H_1(Np^\infty,\mathcal{O}_K)\otimes \mathbb{T}_N/\gothp_{N,r,k} \cong 
 H_1(\mathbb{H}/\Gamma_1(Np^r), \widetilde{L_{k-2}}(\mathcal{O}_K))^{ord}_\gothm.$$

\end{theorem}
\begin{proof} This is \cite[Theorem 1.9]{Hida4}
\end{proof}

As in the previous section, we define a measure sending the open set $(a + p^r\Z_p, a + M\Z)$ in 
$\Z_p^\times \oplus \Z / M\Z$ to 
$U_p^{-r}\{\frac{a}{p^r M}, \infty\}\in H_1(Np^\infty,\mathcal{O}_K).$
This defines an element 
$$L_p(\mathbb{N,\gothm},M)\in H_1(Np^\infty,\mathcal{O}_K)\otimes \Lambda_K [\Z / Mp\Z ^\times].$$
Fix an isomorphism between $\mathbb{T}_N$ and $H_1(Np^\infty,\mathcal{O}_K)$.   This gives
 $$L_p(\mathbb{N,\gothm},M)\in \mathbb{T}_{\gothm,N} \otimes \Lambda_K[\Z / Mp\Z ^\times].$$ \\

By specializing to a classical weight one prime of $\mathbb{T}_{\gothm,N}$
corresponding to an eigenform $f$ we recover
the $p$-adic L-function $L_{M, \gothm, f}$ (up to a $p$-adic unit) described in 
the previous section.  In particular, if $\gothp$ is the classical height one 
prime corresponding to $f$ then the image of $L_p(\mathbb{N,\gothm},M)$ in 
$\mathcal{O}(\gothp) \otimes \Lambda_K[\Z / Mp\Z ^\times]$ is $L_{M, \gothm, f}u$
where $u$ is in $\mathcal{O}(\gothp)^\times$.  \\

\subsection{The one variable $p$-adic $L$-function over the Hida family} \label{one var}
The first variable of $L_p(\mathbb{N,\gothm},M)$ is our Hida family and the second
variable is the cyclotomic variable, which varies over Dirichlet twists
and the different critical values $s=1...k-1$.  
In the previous section, we saw that specializing in the 
first variable recovered the $p$-adic $L$-function of \cite{MTT}.  If we 
specialize in the second variable, we recover a $p$-adic $L$-function
that interpolates a fixed special value over our Hida family.  This $L$-function
wont be meaningful for small classical weights in $\mathbb{T}$, because the classical
cyclotomic $p$-adic $L$-function only interpolates $s$-values less than the weight.\\

Let $\chi$ be a Dirichlet character of level $N=p^rM$ with
$N$ prime to $p$.  
Then $\chi$ gives a homomorphism $\Z / N\Z ^\times \to \mathbb{C}_p$ and a homomorphism
$\Lambda_K[\Z / Mp\Z ^\times] \to \mathbb{C}_p$.  The prime $\gothp_{\chi,s}$ of
$\Lambda_K[\Z / Mp\Z ^\times]$ that is the kernel of this map corresponds to
twisting the $L$-function by $\chi$ and evaluating that $L$-function at $s=1$. \\

The image of $L_p(\mathbb{N,\gothm},M)$ along the map 
$\mathbb{T}_{N,\gothm} \otimes \Lambda_K[\Z / Mp\Z ^\times] \to 
\mathbb{T}_{N, \gothm} \Lambda_K[\Z / Mp\Z ^\times] / \gothp_{\chi,s}$
gives a one variable $p$-adic $L$-function 
$L_p(\mathbb{N,\gothm},\chi) \in \mathbb{T}_{N, \gothm} \otimes \mathcal{O}_K[\chi]$.  
This
$L$-function interpolates the values of the $L$-functions of the cusp forms in
our Hida family for a fixed twist.

\begin{theorem} Let $\gothp \in \Spec (\mathbb{T}_{N, \gothm})$ be a classical height one
prime corresponding to the modular form $f$.  The image of $L_p(\mathbb{N,\gothm},\chi)$
in $\mathbb{T}_{N, \gothm}/\gothp$ is $L^{alg}(f,\chi,1)u$, where $u$ is a $p$-adic unit.

\end{theorem}

\section{Some Geometric Preliminaries}
Let $C_1$ and $C_2$ two connected components of $\Spec (\mathbb{T}_{N, \gothm})$.  It
is often the case that structure maps (i.e. the map onto the weight space)
$\pi_i$ from $C_i$ to $\Spec(\Lambda)$ are isomorphisms. 
 This is
an ideal situation, as functions on $\Spec(\Lambda_K)$ are easily
understood through the Weierstrass preparation theorem.  However,
there are many examples of components whose structure maps are 
not isomorphisms (e.g. families of CM forms where the class group
of the imaginary quadratic field is divisible by some power of $p$).
In this section we address this phenomena by finding formal models
of small affinoid neighborhoods of $C_i^{rig}$ whose coordinate
ring is isomorphic to $\mathcal{O}_K\ang{T}$.  We choose these formal models 
in a way so that they still carry information about congruences between
cusp forms.  We address this in the first subsection.  In the second
subsection we introduce an auxiliary $p$-adic metric
on the $\mathcal{O}_K$ points of a schemes over $\mathcal{O}_K$.  There is nothing
particularly novel here, but it will be convenient for later arguments.  
Finally, we define
intersection multiplicities and explain how
we can pass from schemes to rigid varieties.
There
is no extra difficulty for us to work in a general geometric situation
 and we do so.  \\

Throughout this section we set $\pi:X \to \Spec(\Lambda_K)$
be a finite morphism.  We will assume that $X$ is affine with coordinate
ring $R$.  By a component of $X$ we refer to a subscheme $C:=\Spec(R/\gotha)$ of $X$
where $\gotha$ is a minimal ideal of $X$.

\subsection{The inverse function theorem for formal models} \label{zooming in}
Let $C$ be a component of $X$ and let $x\in C$ be a $\mathcal{O}_K$ point.  We will
assume that $\pi$ is etale at $x$.  

\begin{lemma} \label{p-adic inverse function thm} There is an 
affinoid neighborhood $U$ of $x$ in $\rig{C}$ and  $V$ of $\kappa:=\pi(x)$ in
$\rig{\Lambda}$ such that $U$ and $V$ are isomorphic as rigid varieties.
In particular, if $\gothp_x$ is the maximal ideal defining $x$, the affinoid
$U$ may be taken to be $\{y\in C| |f(y)|\leq \epsilon, y\in \gothp_x\}$.
\end{lemma}

\begin{proof} This is a rigid analytic inverse function theorem.  It can
be deduced from the results of
Chapter III section 9 in \cite{Serre}.
\end{proof}

Since $C$ is finite over $\Spec(\Lambda)$,
we may write $C=\Spec (A)$ with  
$A= \mathcal{O}_K[[T_0]][T_1,...T_n]/I$.  After a change of variables, we may assume that
$x$ corresponds to the point $T=T_i=0$.  By Lemma
\ref{p-adic inverse function thm} there exists $N>0$ such that
the affinoid variety $X=\Sp( (A\otimes K) \langle \frac{1}{p}^{N} T_i\rangle )$ maps 
isomorphically onto $\Sp( K \langle \frac{1}{p}^{N}T_0\rangle )$.  
Set $Y_i=\frac{1}{p}^{N}T_i$ so that $\Spec(A\langle Y_i\rangle )$ is a formal
model for $X$.  Then
$Y_i = f_i(Y_0)$ for some $f_i \in K\ang{Y_0}$ (these power series will
define the inverse map $\Sp(K \ang{Y_0}) \to X$).  The coefficients of $f_i$
tend to zero so we may find $k$ such that $p^k f_i \in \mathcal{O}_K \ang{Y_0}$
for each $k$.  In particular we find $f_i \in \mathcal{O}_K \ang{\frac{1}{p}^kY_0}$
and we may assume that $f_i$ has all integral coefficients by replacing $Y_0$
with $\frac{1}{p}^k Y_0$.  \\

We now have an explicit description of $X$ as $\Sp( K\ang{Y_i} / (Y_i - f_i(Y_0)))$.  
The ring $A\ang{ Y_i}$ injects into $(A \otimes) K \ang{Y_i} \cong  K\ang{Y_i} / (Y_i - f_i(Y_0))$. 
Since the power series $f_i$ have integral coefficients, we have an isomorphism
$A\ang{Y_i} \cong \mathcal{O}_K \ang{Y_i} / Y_i-f_i(Y_0)$ and the ring 
 $\mathcal{O}_K \ang{Y_i} / Y_i-f_i(Y_0)$ is isomorphic to $\mathcal{O}_K\ang{Y_0}$.
 This is more or less an inverse function theorem for formal models.
 
 \begin{lemma} \label{formal inverse function theorem}
 Write $C=\Spec (A)$ with  
$A= \mathcal{O}_K[[T_0]][T_1,...T_n]/I$ as above.  For some $N>0$ 
there exists an isomorphism of formal schemes
\[\Spec(A\ang{\frac{1}{p}^N T_i}) \to \Spec(\mathcal{O}_K\ang{\frac{1}{p}^N T_0}).\]
The isomorphism is given by projecting onto the $T_0$ coordinate.
\end{lemma}
 
The particular situation we are interested in involves two components $C_1$ and
$C_2$ of $X= \Spec(A)$ corresponding to the minimal primes $\gotha_1$ and $\gotha_2$.  
 Let $x_1$ (resp $x_2$) be a $\mathcal{O}_K$-point of $C_1$ (resp $C_2$) 

\subsection{$p$-adic Distances and Congruences}
We start by discussing $p$-adic distances for affine schemes
that are quotients of rings of power series.
We work in this generality because it adds no
extra difficulty.  First let's consider an open $n$-dimensional 
ball $B$ centered around 
the origin of radius $p^{-1}$.  The ring of analytic functions 
converging on this ball is the Tate algebra 
$\mathcal{O}_K[[p^{-1}T_1,...,p^{-1}T_n]]$.  If
$(x_1,...,x_n)$ and $(y_1,...,y_n)$ are two $\mathcal{O}_K$ points in $B$, a natural choice
for the distance between them is $\max |x_i-y_i|_p$.  
It would be great to translate this definition into something more 
intrinsic.  Let $A$ be some reduced 
quotient of a Tate algebra over $\mathcal{O}_K$.

\begin{definition} \label{valuation of ideals}
Let $R$ be a ring of (topologically) finite type over $\mathcal{O}_K$.  Let $I$ be an 
ideal of $R$.  For any prime $\gothp \in \Spec(R)$ let $I(\gothp)$ be the ideal
$I+\gothp \mod \gothp$ in $R/\gothp$.  If $\gothp$ is a prime that is maximal in
$R\otimes \Q_p$, then $R/\gothp$ is a discrete valuation ring (it is a finite extension
of $\Z_p$.)  Let $|I|_{\gothp}$ be the largest absolute value occurring in $I(\gothp)$,
where the absolute value is normalized so $|p|_\gothp=p^{-1}$.

\end{definition}

\begin{definition} Let $\gothp_1$ and $\gothp_2$ be height one prime ideals
with residue characteristic $0$.  We define $d(\gothp_1,\gothp_2)$ to be
$|\gothp_1 |_{\gothp_2}$.

\end{definition}

\begin{lemma} \label{metric} The following properties of $d(,)$ hold:
\begin{enumerate}
\item $d(\gothp_1,\gothp_2)= d(\gothp_2,\gothp_1)$
\item $d(\gothp_1,\gothp_1)=0$
\item Let $(f_1,...,f_r)$ be any set of generators of $\gothp_1$.  Then
$d(\gothp_1,\gothp_2) = \max |f_i \mod \gothp_2|_p$.

\item Suppose $A=\mathcal{O}_K[[T_1,...,T_n]]$.  Assume also that $\gothp_1$ 
corresponds to $(x_1,...,x_n)$ and $\gothp_2$ corresponds
to $(y_1,...,y_n)$.  Then $d(\gothp_1,\gothp_2)\max |x_i-y_i|_p$.

\item Suppose we have a closed embedding 
$f: \Spec(A) \to \Spec(\mathcal{O}_K[[T_1,...,T_n]])$ (so that $A$ is a quotient of
a Tate algebra.)    Then $d(\gothp_1,\gothp_2) = d(f(\gothp_1),f(\gothp_2))$.

\item The non-Archemedian triangle inequality holds.  That is
$d(\gothp_1,\gothp_3)\leq \max( d(\gothp_1,\gothp_2),d(\gothp_2,\gothp_3)).$

\end{enumerate}
\end{lemma}

Since $\mathcal{O}_K[[T_1,...,T_n]]$ is local with maximal ideal $(\pi_K,T_1,...,T_n)$,
we see that $A$ is also local with a maxiaml ideal $\gothm$.  There is an equivalent
definition of distance, which easily connects to congruences of cusp forms.

\begin{lemma} Suppose 
\[ A/ \cong  \mathcal{O}_K \cong A/\gothp_2. \]
Let $r$ be the largest integer such that 
$\gothm^r+\gothp_1 = \gothm^r+\gothp_2$.  Then $p^{\frac{-r}{e}}=d(\gothp_1,\gothp_2)$,
where $e$ is the ramification index of $\mathcal{O}_K$ over $\Z_p$.  In 
particular, the natural map 
$$A\to A(\gothp_1)\cong \mathcal{O}_K \to \mathcal{O}_K/\pi_K^r$$ 
is the same as 
$$A\to A(\gothp_2)\cong \mathcal{O}_K \to \mathcal{O}_K/\pi_K^r,$$
and $r$ is the largest integer for which this is true.
\end{lemma}

\begin{proof} Not very hard.
\end{proof}

Now consider the big Hecke algebra $\mathbb{T}_N$ described in \ref{Hida theory}.
Let $\gothm$ be a maximal ideal of $\mathbb{T}_N$.  The ring $\mathbb{T}_{N,\gothm}$
is local, reduced, finite over $\Lambda$ and $\gothm$-adically complete.  
We will assume that the residue fields of $\mathbb{T}_{N,\gothm}$ and $\Lambda$ are
the same, which is equivalent to saying $\mathbb{T}_{N,\gothm}/\gothm \cong \mathcal{O}_K /
\pi_K \mathcal{O}$ 
(if this is not the case we may replace $K$ with a larger extension).
Let $r_1,...,r_n$ generate $\mathbb{T}_{N,\gothm}$ as a $\Lambda$ algebra.  If
$r_i$ is a unit, then it is equivalent to an element of $\mathcal{O}_K\ang{T_0}$ modulo $\gothm$,
so we may assume that each $r_i$ is in $\gothm$.  Thus the $r_i$ are topologically
nilpotent and we have a surjection $\mathcal{O}_K[[T_1,...,T_n]]$ by sending $T_i$ to $r_i$.
Geometrically, we can embed the local components of our big Hecke algebra into 
an $n$-dimensional open unit Ball.  Summarizing everything gives:

\begin{proposition} \label{distances and congruences} The local big Hecke algebra $\mathbb{T}_{N,\gothm}$ embeds
into an $n$-dimensional open unit ball.  This mapping is "isometric" with
respect to the natural $p$-adic metric on the $p$-adic ball and the metric
$d(,)$ defined above.  Let $\gothp_1,\gothp_2\in\Spec( \mathbb{T}_{N,\gothm})$
be $\mathcal{O}_K$-points.  If
$d(\gothp_1,\gothp_2)=p^{\frac{-r}{e}}$ then the modular forms (not necessarily of 
classical weight) corresponding to $\gothp_i$ are congruent modulo $\pi_K^r$
but not $\pi_K^{r+1}$.
\end{proposition}

\subsection{Intersection Multiplicities}
In this section we will define the intersection multiplicity of two crossing
components of a curve over $\mathcal{O}_K$.  After proving some basic properties we will
reduce our problem to talking about $f(C_i)$ as in the previous section.
Let $X$ be a scheme over $\mathcal{O}_K$ such that $\rig{X}$ has dimension one.  Let
$C_1$ and $C_2$ be connected components of $X$ and let 
$x$ be an $\mathcal{O}_K$ point of $X$ 
whose residue characteristic is $0$ (so we may think of $x$ as a
point in $\rig{X}$.)  Let $j_i:C_i\to X$ be the natural inclusion.
Let $\mathcal{I_i}$ be the 
$\mathcal{O}_X$-sheaf of ideals that define the component $C_i$.  

\begin{definition} 
The intersection multiplicty $I(X, C_1,C_2,x)$ of $C_1$ and $C_2$ at 
$x$ is the $K$
dimension of 
$$(\mathcal{O}_{C_1}/j_1^*(I_2))_x = (\mathcal{O}_X/(\mathcal{I}_1+\mathcal{I}_2))_x
=(\mathcal{O}_{C_2}/j_2^*(\mathcal{I}_1))_x.$$
\end{definition}

\begin{remark} The intersection multiplicity is nonzero if and only if
both $C_1$ and $C_2$ contain $x$.
\end{remark}

Since this definition is Zariski local (and rigid analytic local, we shall see...)
we may take an affine neighborhood $U=\Spec(A)$ of $x$.  Let $\gotha_1$ and $\gotha_2$ be
the minimal prime ideals of $A$ defining the components $C_1$ and $C_2$.  Let $p_x$
be the prime corresponding to $x$.  Then 
$$I(X,C_1,C_2,x)= \dim_{\Q_p}(A/(\gotha_1+\gotha_2))_{p_x}.$$

\begin{lemma}
The following properties of intersection multiplicities are true.

\begin{enumerate}
\item Let $X'$ be another formal model of $\rig{X}$.  Let $C_1'$ and $C_2'$
be connecting components corresponding to $\rig{C_1}$ and $\rig{C_2}$.  These
components cross at $x'$ whose image in $\rig{X}$ is $x$.  Then 
$I(X, C_1, C_2, x)= I(X', C_1', C_2', x)$.  
In other words, the intersection multiplicity
only depends on the rigid fiber.

\item Recall how stalks are defined for a sheaf on a 
rigid analytic varieties
$$\mathcal{F}_{\rig{X},x} = \mathop{\varinjlim_{x \in U}}_{U \text{ is affinoid}} 
\mathcal{F}(U).$$   Then 
$I(X, C_1,C_2,x) = \dim_{Q_p} (\mathcal{O}_{\rig{X}}/ 
(\rig{\mathcal{I}_1} + \rig{\mathcal{I}_2}))$.  In other words, we can
use the rigid analytic stalks or the Zariski stalks to find intersection
numbers.  

\item Let $U$ be an affinoid neighborhood of $x$.  Let $\goth{U}$ be
a formal model of $U$ and let $\goth{C}_i$ be the formal models of 
$U\cap \rig{C_i}$ that are components of $\goth{U}$.  
Then $I(X, C_1, C_2, x)=I(\mathfrak{C}_1,\mathfrak{C}_2, x)$.

\item Let $i:Y \to X$ be a closed subscheme such that $i(Y)$ contains the 
generic points of $C_1$ and $C_2$.  Then $I(X, C_1, C_2, x)$
is the same as $I(Y, Y \times_X C_1, Y \times_X C_2, i^{-1}(x))$.

\end{enumerate}

\end{lemma}

\begin{proof} The first statement is true because $p$ is invertible in
$\mathcal{O}_{X,x}$.  The third statement is an immediate consequence of the
second statement.  To prove the second statement, pick an affine 
neighbodhood of $x$ and use the fact that 
$\widehat{\mathcal{O}_{X,x}}\cong\widehat{\mathcal{O}_{\rig{X},x}}$, where $\hat{A}$
denotes the completion of a local ring $A$ along it's maximal ideal 
(see \cite{bosch}.)  The last statement is easily checked by picking an
affine neighborhood of $x$.

\end{proof}

In the simple situation of two rational curves crossing in $\mathcal{O}_K \ang{T_0, T_1}$,
we can come up with a precise formula for the intersection number.  Let $X$ be
finite over $\Spec(\mathcal{O}_K \ang{T_0})$ and let 
$X \to \mathcal{O}_K \ang{T_0, T_1}$ be a closed embedding.  Let $C_1$ and $C_2$ be
two connected components of $X$ that are isomorphic both $\mathcal{O}_K \ang{T_0}$.
Then we have embeddings $C_i \to \Spec(\mathcal{O}_K \ang{T_0, T_1})$ that factor through
$X$ and we see that $C_i = \Spec(\mathcal{O}_K \ang{T_0}{T_1} / (T_1 - f_i(T_0))$ with
$f_i(T_0) \in \mathcal{O}_K \ang{T_0}$.  Then we can compute the intersection number:

\begin{lemma} \label{Intersection lemma}
The intersection $I(X, C_1, C_2, x)$ is equal to the largest power of $T_0$
dividing $f_1-f_2$.
\end{lemma}

\section{Crossing components in Hida families} \label{main proof}
We are now ready to prove Theorem \ref{Crossing Theorem}.
Recall our definition of $\mathbb{T}_{N,\gothm}$ from Section \ref{Hida theory}
as a local component of
the Hecke algebra that acts on the space of all cusp forms of tame level $N$.
There is a map 
\[\pi: \Spec(\mathbb{T}_{N,\gothm}) \to \Spec(\Lambda_k) \cong
\Spec(\mathcal{O}_K [[T]]). \]
We will assume that $\Spec(\mathbb{T}_{N,\gothm})$ has at least two components
$C_1$ and $C_2$.  Let $\kappa$ be a $\mathcal{O}_K$-point of $\Spec (\Lambda_K)$
that is the $p$-adic limit of classical weights and let $x_1$ (resp $x_2$)
be a $\mathcal{O}_K$-point of $C_1$ (resp. $C_2$) with $\pi(x_1)=\kappa$.  After
a change of variables we may take $\kappa$ to be the point $T=0$.  We
will assume that the restriction of $\pi$ to $C_1$ (resp $C_2$) is etale
at $x_1$ (resp $x_2$).
If $C_1$ and $C_2$ cross above $\kappa$ then it is possible to choose $x_1$
and $x_2$ to be the same point when viewed as points of $\Spec(\mathbb{T}_{N,\gothm})$. \\

\subsection{Reducing to the simplest geometric situation}

We may write $C_i = \Spec(A_i)$ where
 $A_i$ is generated as a $\Lambda_K$-algebra by $T_{i,1},...,T_{i,n}$ and
 we may choose the $T_{i,1}$ so that $x_i$ corresponds to the origin.
By applying Lemma \ref{formal inverse function theorem} to $C_1$ and $C_2$ simultaneously
we know that there exists $N>0$ such that 
\[B_i:= A_i\ang{\frac{1}{p}^N T_0,\frac{1}{p}^N T_{i,1},...,\frac{1}{p}^NT_{i,n}}
\cong \mathcal{O}_K\ang{\frac{1}{p}^N T_0}. \]
Let $Y = \frac{1}{p}^N T_0$.  
Then $\Spec(B_i)$ is a connected component of $\Spec(\mathbb{T}_{N,\gothm} \ang{Y})$.
To see this, consider the commutative diagram
\begin{center}
$\begin{CD}
\Spec(B_i)
 @>>> \Spec(\mathbb{T}_{N,\gothm} \ang{Y}) \\
@A \pi_i^{-1} AA @VVV\\
\Spec(\mathcal{O}_K \ang{Y}) @= \Spec(\mathcal{O}_K \ang{Y})
\end{CD}$
\end{center}
From this diagram we see that the generic point of $\Spec(B_i)$ must be sent
to a minimal prime in $\Spec(\mathbb{T}_{N,\gothm}\ang{\frac{1}{p}^N})$ and that
the map $\mathbb{T}_{N,\gothm}\ang{\frac{1}{p}^N Y} \to B_i$ is surjective.  
Let $Z$ be the scheme theoretic union of $\Spec(B_i)$ inside of
$\Spec(\mathbb{T}_{N,\gothm} \ang{\frac{1}{p}^N})$.  Then $Z$ comes
naturally equip with a map to $\Spec(\mathcal{O}_K \ang{\frac{1}{p}^N Y})$,
which we call $\pi$ by abuse of notation.  
This map is surjective and
finite of degree two by our definition of $Z$.  Thus 
\[ Z=\Spec(\mathcal{O}_K \ang{\frac{1}{p}^N Y}[T]/f(T,Y) ),\]
where $T$ is some indeterminate and $f$ is monic of degree two
in $T$.  As $\pi$ admits two sections (one for each $\Spec(B_i)$)
the polynomial $f$ factors into linear terms, i.e.

\[f(T,Y)=(T-g_1(Y))(T-g_2(Y)). \]
Here we have $g_i \in \mathcal{O}_K \ang{\frac{1}{p}^N Y}$
and $g_i$ corresponds to the closed subscheme $\Spec(B_i)$.  
The points $x_1$ and $x_2$ are the same if and only if
$g_1$ and $g_2$ have the same constant term.  \\

There is a natural map $s: Z \to \mathbb{T}_{N,\gothm}$.  For
any two $\mathcal{O}_K$-points $y_1,y_2$ in $Z$ we have a relation between
the distances of these points in the two spaces:
\[ d(y_1,y_2) = p^{N}d(s(y_1),s(y_2)).\] 
These points correspond to cusp forms $f_{y_1}$ and $f_{y_2}$.
We combine Lemma \ref{distances and congruences} with this relation to get:

\begin{lemma} \label{congruences in our model}
If
$d(y_1,y_2)=p^{\frac{-r}{e}-N}$ then the cusp forms $f_{y_1}$ and
$f_{y_2}$ are congruent modulo $\pi_K^r$
but not $\pi_K^{r+1}$.  Informally, the distances between points in $Z$
tells us exactly how congruent the corresponding cusp forms are.
\end{lemma}

\subsection{An ideal of differences of $L$-values}

Recall that for $\chi$ and $k\in \mathbb{N}$ there is a $1$-variable $p$-adic $L$-function
$L_p(\mathbb{T}_{N,\gothm}, \chi) \in \mathbb{T}_{N,\gothm} [\chi]$.  Let
$L_p(B_i, \chi)$ be the restriction of this function to $\Spec(B_i)$.  In
particular $L_p(B_i, \chi)$ is in $B_i$.  Then for 
$x \in \Spec(B_i)$ corresponding to a classical modular form $f_x$
we see that $L_p(B_i, \chi)$ evaluated at $x$ is equal to the algebraic part of
$L(f_x, \chi, 1)$.  The isomorphism
$\pi_i: \Spec( B_i[\chi]) \to \Spec( \mathcal{O}_K[\chi] \ang{\frac{1}{p}^N T})$ allows
us to view $L_p(B_i, \chi)$ as a power series in $T$.  In fact,
this gives us the Taylor series expansion o $L_p(C_i, \chi)$ expanded
around $x_i$.  

\begin{definition} Let $I_L$ be the ideal of $\mathcal{O}_K^{cyc} \ang{\frac{1}{p}^N T}$
generated by the elements $L_p(B_1, \chi) - L_p(B_2, \chi)$ for all Dirichlet characters $\chi$.
\end{definition}

\begin{lemma} \label{L-ideal} 
Let $\kappa \in \Spec(\Z_p\ang{T})$ corresponds to a classical
weight.  Then $v_p(I_L(\kappa))$ is equal to the valuation of $g_1(\kappa)-g_2(\kappa)-N$.
\end{lemma}
\begin{proof} Let $y_i$ be the point of $\Spec(B_i)$ lying above $\kappa$.  By  
Proposition \ref{p-adic cong} we know that $v_p(I_L(\kappa))$ tells us exactly how
congruent the modular forms $f_{y_i}$ associated to these points are.  Then by
applying Lemma \ref{congruences in our model} we see that $v_p(I_L(\kappa)) = 
d(y_1,y_2)p^{-N}$.  The coordinates of $y_i$ are $(\kappa, g_i(\kappa))$.  Then
applying  parts four and five of Lemma \ref{metric} we find that 
$d(y_1,y_2)=|g_1(\kappa)-g_2(\kappa)|$.  The result follows.
\end{proof}

\subsection{Proof of Theorem \ref{Crossing Theorem}}
We are now ready to prove Theorem \ref{Crossing Theorem} by comparing
$g_1(Y)-g_2(Y)$ with the ideal $I_L$.  The proof involves
using Lemma \ref{L-ideal} to see how congruences behave
as we approach the crossing point at classical weights.

\begin{definition} Define $B_{p^r}$ to be 
$\Spec(\Z_p\ang{\frac{1}{p}^r Y})$, the neighborhood of 
radius $\frac{1}{p}^r$ around $\kappa=0$.
\end{definition}

\begin{lemma} \label{restrict poly}  Let $f(T)\in\Z_p\ang{Y}$.  Let $x\in \Z_p$ with $v_p(x)<1$.  If
$f(x) \neq 0$ (equivalently $T-x$ does not divide $f(T)$), then there exists
a ball $B_{p^r}(x)$ centred at $x$ such that $f(T)$ is equal to a power of
$p$ times a unit when restricted to $B_{p^r}(x)$.  Equivalently, we have
$f(T)=p^s u(T)$ with $u(T)$ being in $\Z_p[[p^r (T-x)]]^\times$.
Any neighbourhood of $x$ where $f(T)$ has no roots will suffice.
\end{lemma}

\begin{proposition} The largest power of $Y$ dividing the 
ideal $I_L$ is $I(\mathbb{T}_N, C_1, C_2, x_1)$.  
\end{proposition}

\begin{proof} 
Let $n= I(\mathbb{T}_N, C_1, C_2, x)$ and let $m$ be the largest power of $Y$
contained in $I_L$.  
Let $\chi$ be a Dirichlet character and $k \in \mathbb{N}$.
By
Lemma \ref{restrict poly} and Lemma\ref{Intersection lemma} 
we may replace $\mathcal{O}_K \ang{Y}$
with a smaller ball $B_{p^r} (0)$ where 
\begin{eqnarray}
g_1(Y)-g_2(Y)&=& Y^nu(Y)\pi_K^r  \\
L_p(B_1,\chi)-L_p(B_2,\chi) &=& Y^{m_{\chi,s}} v_{\chi,s}(Y) c_{\chi,s}.
\end{eqnarray}
Here $u(Y), v_{\chi,s}(Y)$ are units and $c_{\chi,s}$ is a constant in $\mathcal{O}_K[\chi]$.
Note that $m\min( m_{\chi,s})$.
 Pick a sequence $t_n$ of points in $B_{p^r}(0)$
that converge to $0$ such that each $t_n$ corresponds to a classical weight.  We denote
by $L_p(B_i,\chi)(t_n)$ the function $L_p(B_i,\chi)$ evaluated at the point $t_n$.
Then by Lemma \ref{L-ideal} we know that

\begin{eqnarray}
 v_p(g_1(t_n) - g_2(t_n)) &\leq & v_p(L_p(B_1,\chi)(t_n)-L_p(B_2,\chi)(t_n)) \text{ which gives} \\
 nv_p(t_n)  + r & \leq & m v_p(t_n) + t.
\end{eqnarray}
As $v_p(t_n) \to \infty$ as $n\to\infty$ we see that $n\leq m_{\chi,s}$ and so 
$n\leq m$.  In particular we
find that $Y^n$ divides $I_L$.  Conversely, we know
\[ mv_p(t_n) \leq m_{\chi, s}v_p(t_n)\leq  v_p(L_p(B_1,\chi)(t_n)-L_p(B_2,\chi)(t_n)). \]
Applying Lemma \ref{L-ideal} again while $n\to \infty$ shows $m\leq n$.

\end{proof}

The proof of Theorem \ref{Crossing Theorem} follows as a corollary.

\section{Some examples}
In this section we explore two situations where crossings may occur.  First,
we look at two Hida families of different levels
with the same residual representation.  Under a suitable hypothesis on
the levels, we can determine if the two families will cross in a higher level
by looking at the $L$-functions on each family modified by appropriate Euler factors.
These modified $L$-functions were introduced in \cite{EWP} for trivial tame character
using results from \cite{Wiles}.

\subsection{Components of different level}
In this subsection we apply our results on crossing components to
the situation described in Section 2.6 of \cite{EWP}.  We will
briefly summarize the set up.  For details
and references see \cite{EWP}.
Let $\overline{\rho}$ be a modular residual Galois representation and let
$\overline{V}$ be the $\mathbb{F}_q$ vector space on which
$G_\mathbb{Q}$ acts.  We will 
assume that $\overline{\rho}$ is odd, irreducible, 
$p$-ordinary, and $p$-distinguished.  Fix a $p$-stabilization of $\overline{\rho}$.
We may assume that $\mathbb{F}_q$ is the possible extension of
$\mathbb{F}_p$ (i.e. $\mathbb{F}_q$ is generated by the traces
of $\overline{\rho}$.  Let $N(\overline{\rho})$ be the conductor of
$\overline{\rho}$.  For a prime $l\neq p$
let $n_l$ be the dimension of the $I_l$ invariant of $\overline{V}$.  
Let $\Sigma$ be a finite set of primes not containing $p$.  Define
\[ N(\Sigma) := N(\overline{\rho})\Pi_{l\in\Sigma} l^{n_l}. \]
For any tame level $N$ we let $\mathbb{T}_N'$ be the Hecke algebra
acting on $S(Np^\infty, \mathcal{O}_K)$ generated by $U_p$
and $T_l$ for all $l\not |Np$ (explicitly, we are just leaving out 
the Atkin Lehner operators).  

We let
$\mathbb{T}^{new}_N$ denote the Hecke algebra generated by
$T_l$ for primes $l\not | Np$ and $U_l$ for $l|Np$ acting
on the subspace of $S(Np^\infty, \mathcal{O}_K)$ consisting
of all newforms.  Then we have a natural map of $\Lambda$-algebras
\[\mathbb{T}_{N(\Sigma)}' \to \Pi_{M|N(\Sigma)} \mathbb{T}^{new}_M. \]
This map becomes an isomorphism after tensoring over $\Lambda$ 
with its fraction field $\mathcal{L}$.
As described by Hida \cite{Hida3} there is a Galois representation
$\rho'_M:G_\mathbb{Q} \to \mathbb{T}^{new}_M \otimes \mathcal{L}$ for any $M$.
This gives a Galois representation 
$\rho': G_\mathbb{Q} \to \mathbb{T}_{N(\Sigma)}' \otimes \mathcal{L}$.
We have the following two theorems

\begin{theorem} \label{Wiles and Diamond} There exists a unique maximal prime $\gothm$ of
$\mathbb{T}_{N(\Sigma)}'$ such that the residual representation of
the composition
\[ G_\mathbb{Q} \to \mathbb{T}_{N(\Sigma)}' \to \mathbb{T}_{N(\Sigma),_\gothm} \]
is $\overline{\rho}$.  Furthermore, there is a unique maximal prime 
$\gothn$ of $\mathbb{T}_{N(\Sigma)}$ such that the two local Hecke
algebras are isomorphism:

\[ \mathbb{T}_{N(\Sigma),\gothn} \cong \mathbb{T}_{N(\Sigma), \gothm} '.\]
\end{theorem}

\begin{proof} Cite Wiles and Diamond.
\end{proof}

\begin{theorem} Let $\gotha$ be a minimal primes ideal of $\mathbb{T}_{N(\Sigma), \gothn}$.
There exists some $N(\gotha)|N(\Sigma)$ and a minimal prime ideal $\gotha'$ of 
$\mathbb{T}^{new}_{N(\gotha)}$ that makes the following diagram commute.  

\begin{center}
$\begin{CD}
\mathbb{T}_{N(\Sigma),\gothn} \cong \mathbb{T}_{N(\Sigma), \gothm} '
 @>>> \Pi_{M|N(\Sigma)} \mathbb{T}^{new}_M \\
@VVV @VVV\\
\mathbb{T}_{N(\Sigma),\gothn} / \gotha @>>> \mathbb{T}^{new}_{N(\gotha) }/\gotha'
\end{CD}$
\end{center}

\end{theorem}
\begin{proof} This follows from \ref{Wiles and Diamond} and the isomorphism
\[ \mathbb{T}_{N(\Sigma)}' \otimes \mathcal{L} \to 
\Pi_{M|N(\Sigma)} \mathbb{T}^{new}_M  \otimes \mathcal{L}. \]
See Proposition 2.5.2 in \cite{EWP} for more details.
\end{proof}

\begin{remark} We may think of $\Spec(\mathbb{T}_{N(\Sigma),\gothn} / \gotha)$ as a family
of old forms of level $N(\Sigma)$ and $\Spec(\mathbb{T}^{new}_{N(\gotha) }/\gotha')$ as
a family of new forms of level $N(\gotha)$.  If $x \in \Spec(\mathbb{T}_{N(\Sigma), \gothn})$
corresponds to the classical old form $f_x$, then there is a corresponding 
$x' \in \Spec(\mathbb{T}^{new}_{N(\gotha) }/\gotha')$ that is sent to $x$ under the map
$\Spec(\mathbb{T}^{new}_{N(\gotha) }/\gotha') \to 
\Spec( \mathbb{T}_{N(\Sigma),\gothn} / \gotha ).$
The point $x'$ corresponds to a newform $f_{x'}$ of level $N(\gotha)$.  The Fourier
coefficients of $f_x$ and $f_{x'}$ agree away from the primes dividing the level $N(\Sigma)$.

\end{remark}

By Theorem \ref{Crossing Theorem} we can determine when two components
of $\mathbb{T}_{N(\Sigma)}$ by looking at $p$-adic $L$-functions on each component.  
It is then natural to ask if we can determine when a family of newforms of level $M_1|N(\Sigma)$
will cross a family of newforms of level $M_2|N(\Sigma)$ by looking at $p$-adic $L$-functions.  
To employ Theorem \ref{Crossing Theorem} it is necessary to relate our $p$-adic $L$-functions
on $\Spec(\mathbb{T}_{N(\Sigma),\gothn} / \gotha)$ to our $p$-adic $L$-functions on
$\Spec( \mathbb{T}_{N(\gotha),\gothn} / \gotha' )$.  The former interpolates special
values of eigenforms for the Hecke algebra $\mathbb{T}_{N(\Sigma)}$ and the later interpolates
special values of eigenforms for the Hecke algebra $\mathbb{T}_{N(\gotha)}$.  As these
two Hecke algebras only differ at $l|N(\Sigma)$, it is
natural to suspect that the two $L$-functions will be the same after introducing
some Euler factors for the primes $l|N(\Sigma)$.  

\begin{definition}
Let $l\neq p$ be a prime and let $\chi$ be a Dirichlet character of level $Mp^r$.
Define $E_{N(\gotha)}(\chi, l) \in \mathbb{T}_{N(\gotha)}$ as follows:
\[ E_{N(\gotha)}(\chi, l) := 
\begin{cases}  
1 - \chi (l)T_l l^{-1} + \chi (l^2 )\ang{l} l^{-3} & \mbox{if }   l\not | N(\gotha) \\
1 - \chi (l)T_l l^{-1} & \mbox{if } l | N(\gotha) \\
\end{cases} \]
We then define

\[ E_\Sigma(\gotha, \chi) : = \Pi_{l \in \Sigma} E_{N(\gotha)}(\chi, l). \]
\end{definition}

\begin{remark} This definition is similar to Definition 2.7.1 and 3.6.1 in \cite{EWP}.  
The Euler of \cite{EWP} varies over a branch of the Hida family and the cyclotomic
variable, while our definition only varies over the branch.  If the
conductor of $\chi$ is a power of $p$ then our Euler factor is equal to a specialization 
of the Euler factor defined in \cite{EWP}.  
\end{remark}

\begin{proposition} \label{Euler factor} Let $\chi$ be a Dirichlet character.
There exists a unit $u \in \mathbb{T}_{N(\gotha)}/\gotha '$ independent
of $\chi$ Dirichlet character such that

\[ L_p(\mathbb{T}_{N(\Sigma)}/\gotha, \chi) = E_\Sigma(\gotha, \chi) 
L_p(\mathbb{T}_{N(\gotha)} / \gotha', \chi). \]

\end{proposition}

\begin{proof} The proof follows from the computations in beginning of the proof of 
Theorem 3.6.2 in \cite{EWP}.  The only difference is that we are specializing in the cyclotomic 
variable and we allow a nontrivial tame conductor.
\end{proof}

\begin{theorem} 
Let $M_1$ and $M_2$ be two integers dividing $N(\Sigma)$.  Let $\gotha_i$ 
be a minimal prime ideal of $\mathbb{T}^{new}_{M_i}$.  Let $C_1$ and $C_2$
be the components of $\mathbb{T}_{N(\Sigma)}$ corresponding to
$\gotha_1$ and $\gotha_2$.  The following are equivalent:

\begin{itemize}
\item The components $C_1$ and $C_2$ cross at a point $x$.  We assume that each component
is etale at $x$ over the weight space and the weight $\kappa$ of $x$ is 
the $p$-adic limit of classical
weights.  
\item There exists a point $x_i$ of $\Spec(\mathbb{T}^{new}_{M_i})$ over $\kappa$
and a unit $u$ of $\Lambda$ such that for all Dirichlet characters $\chi$
the value of $u L_p(\mathbb{T}^{new}_{M_1}/\gotha_1, \chi) E_\Sigma (\gotha_1, \chi)$
evaluated at $x_1$ is the same as 
$L_p(\mathbb{T}^{new}_{M_2}/\gotha_2, \chi) E_\Sigma (\gotha_2, \chi)$
evaluated at $x_2$.
\end{itemize}

\end{theorem}
\begin{proof} This is a consequence of Proposition \ref{Euler factor} and
Theorem \ref{Crossing Theorem}.
\end{proof}

\section{Ramification over the weigth space}

In this section we describe how $p$-adic $L$-functions behave when a Hida
family is ramified over the weight space.  Recall
that $\Lambda_K$ is the ring of power series over $\mathcal{O}_K$.  
Let $C$ be a component of $\Spec(\mathbb{T}_N)$.  Then $C$ is affine
and we let $A$ be the coordinate ring.  Informally, the main result of this
section says that $C$ has
ramified points over $\Lambda_K$ if and only if there exists
an $L$-function $L_p(C,\chi)$ that acquires singularities after
being hit with the differential operator $\frac{d}{dT}$.  Here $L_p(C,\chi)$
refers to the $L$-function defined in Section \ref{one var} restricted
to the component $C$.  The
singularities will be at ramified points.

\begin{theorem} \label{ramification theorem} 
A regular $\mathcal{O}_K$-point $x \in C$ is ramified over 
$\pi(x) \in \Spec(\Lambda_K)$
if and only if there exists a Dirichlet twist such
that $\frac{d}{dT} L(C,\chi)$ has a pole at $x$, where $T$ is a
parameter of the weight space.  The ramification
index of $x$ over $\pi(x)$ is equal to one more than the largest order pole occurring.
\end{theorem}

\begin{proof}
First lets assume that $\pi$ is etale at $x$.  Informally, this means that 
a small neighborhood of $x$ looks just like part of $\Spec(\Lambda_K)$, so taking
the derivative with respect to $T$ should not introduce any poles.  
More precisely, let $\widehat{A_x}$ be the
completion along the maximal ideal of the stalk of $\mathcal{O}_{C}$ at $x$.
Define $\widehat{\Lambda_{K, \pi(x)}}$ to be the completion along the maximal ideal of the
stalk of $\mathcal{O}_{\Lambda_K}$ at $\pi(x)$.  The natural map from
$\widehat{\Lambda_{K,\pi(x)}} \to \widehat{A_x}$ is an etale morphism
of complete local rings with isomorphic residue fields.  This means the two rings
are isomorphic.  This isomorphism commutes with the differential operator $\frac{d}{dT}$.
In particular there is a map $A \to \widehat{\Lambda_{K,\pi(x)}}$ such that
commutes with $\frac{d}{dT}$ and the maximal ideal of $\widehat{\Lambda_{K,\pi(x)}}$ pulls
back to $x$.  It is then clear that for any $f\in A$ the function
$\frac{d}{dT} f$ does not have a pole at $x$.  \\

The converse is more difficult.  We begin by making some geometric simplifications
similar to those in Section 6.1.  
Assume that $\pi$ is ramified at $x$.  After
a change of coordinates we may assume that $\pi(x)=0$.  The ring
$A_x$ is a discrete valuation ring since $x$
is a regular point of codimension one.    Let $X$ be a uniformizing element of 
$A_x$.  We may assume that $X$ is in $A$ by clearing any
denominators.  Since $X$ is topologically nilpotent in $A$ we have a map
\[g: C=\Spec(A) \to \Spec(\mathcal{O}_K[[Y]])\]
 induced by the ring map
sending $Y$ to $X$.  This map is etale at $x$ so we may apply Lemma 
\ref{formal inverse function theorem}.  In particular, let $Y_1,...,Y_r$ generate
$A$ as a $\mathcal{O}_K[[Y]]$-algebra.  There exists $N$ large enough
so that
\[ g: \Spec(A \ang{\frac{1}{p}^NY, \frac{1}{p}^N Y_i} \to 
\Spec(\mathcal{O}_K\ang{\frac{1}{p}^NY}) \]
is an isomorphism.  Setting $T'=\frac{1}{p}^NT$ and $Y'=\frac{1}{p}^NY$, 
we have $T'=f(Y')$ where 
$f('Y)\in \mathcal{O}_K\ang{Y'}$.  By increasing $N$,
we may guarantee that the only zero of $f(Y')$ is at $Y'=0$.  Thus
$f(Y')=\pi_k^t u(Y')Y'^e$ where $u(Y')$ is a unit.   We may write
\[ A \ang{\frac{1}{p}^N Y, \frac{1}{p}^N Y_i}  \cong
\mathcal{O}_K\ang{T', Y'} / (\pi_k^t u(Y')Y'^e - T'). \]
The ramification of $x$ over $\pi(x)$ is seen to be $e$.  We also
remark distances of points relate to higher congruences of the corresponding
cusp forms.  If $x_1$ and $x_2$ are two $\mathcal{O}_K$-points of 
$A \ang{\frac{1}{p}^N Y, \frac{1}{p}^N Y_i}$  
corresponding to classical cusp forms $f_{x_1}$ and $f_{x_2}$ that are
congruent modulo $\pi_K^r$ then $d(x_1,x_2)=p^{-\frac{r}{e(K|\mathbb{Q}_p)} - N}$.
This is more or less the same as Lemma \ref{congruences in our model}.

For a Dirichlet character $\chi$ we let $L_p(\chi)$ be the restriction of 
$L_p(C, \chi)$ to 
$A \ang{ Y', T'}[\chi]$.  Then we may think of $L_p(\chi)$ as
an element of $\mathcal{O}_K \ang{Y'}[\chi]$ written as
\[ \Sigma c_{i,\chi} Y'^i. \]
Let $\alpha$ be an $\mathcal{O}_K$-point of our weight space
$\Spec(\mathcal{O}_K\ang{T'})$ corresponding to a classical weight.  
By abuse of notation we will also
think of $\alpha$ as an element of $\mathcal{O}_K$ for the parameter $T'$.  
Let
$a_1$ and $a_2$ be distinct $\overline{\mathcal{O}_K}$-points 
in 
$\pi^{-1}(\alpha) \subset 
A \ang{ Y', T'} \cong \mathcal{O}_K \ang{Y'}.$  By abuse
of notation we we will also
think of $a_i$ as an element of $\overline{\mathcal{O}_K}$ 
for the parameter $Y'$.  Both $a_1$ and $a_2$
are roots of $\alpha - \pi_k^t u(Y')Y'^e$ and $d(a_1,a_2)= |a_1-a_2|_p$.  
Relating $d(a_1,a_2)$ to congruences and then applying Theorem \ref{p-adic cong}
gives 
\begin{eqnarray}
\log_p d(a_1,a_2)) & = & v_p(a_1-a_2) \\
& = & \min_{\chi} (v_p(L_p(\chi)(a_1) - L_p(\chi)(a_2)) + N) \\
& = & \min_{\chi} (-N + v_p(\Sigma c_{i,\chi}(a_1^i - a_2^i))).
\end{eqnarray}
If we take $\alpha$ to have large enough valuation then both of the $a_i$'s
have valuation larger than $N+1$
(look at the Newton polygon of $\alpha - \pi_k^t u(Y')Y'^e$).  
This means $v_p(a_1^i - a_2^i)>v_p(a_1 - a_2) + N+1$ whenever $i>1$.  
\begin{eqnarray}
v_p(a_1-a_2) & = & \min_{\chi} (-N + v_p(\Sigma c_{i,\chi}(a_1^i - a_2^i))) \\
& = & \min_{\chi} (-N + v_p(c_{1,\chi}(a_1 - a_2))) \\ 
& = & \min_\chi (-N + v_p(c_{1,\chi}) + v_p(a_1 - a_2)).
\end{eqnarray}
Thus there exists a Dirichlet character $\chi_0$ 
such that $c_{1,\chi_0} \neq 0$.  
Differentiating the equation \[T-\pi_K^r u(Y)Y^e=0\] with respect to $T$ 
yields 
\[ \frac{dY}{dT} = \frac{1}{\pi_K^rY^{e-1}}\frac{1}{Yu'(Y) + eu(Y)}. \]
This shows that $\frac{dY}{dT}$ has a pole at $Y=0$ of order $e-1$.  
Since $L_p(\chi)$ has a nonzero linear term we find that
$\frac{d}{dT}L_p(\chi)$ has a pole of order $e-1$.

\end{proof}

For the previous result, we choose a parameter for the weight space.
The result holds true for any parameter and it would be nice to have
a statement that makes no reference to any choice of parameter.  This
can be achieved using the Guass-Manin connection, which can be
defined without choosing a basis.  For an overview of
the Gauss-Manin connection see \cite{katz-oda} or \cite{katz-thesis}.  
More precisely,
consider the relative $0$-th de Rham cohomology group $H_{dR}^0(C/\Spec(\Lambda_K))$
(see for example \cite{Grothendieck}).  We may identify
$H_{dR}^0(C/\Spec(\Lambda_K))$ with $\pi_*(\mathcal{O}_C)$ (here $\mathcal{O}_C$
just denotes the structure sheaf of $C$).  Let $U$ be an open subscheme of
$C$ such that $\pi|_U$ is etale.  Then following \cite{katz-oda}
there is a Gauss-Manin connection 
\[ \triangledown : H_{dR}^0(C/\Spec(\Lambda_K))|_U \to H_{dR}^0(C/\Spec(\Lambda_K)) 
\otimes \Omega_{\Spec(\Lambda_K)}|_U. \]
If $f \in \Gamma( \mathcal{O}_C, U)$ and $T_0$ is any parameter of the weight
space then $\triangledown(f) = \frac{d}{dT_0}f dT_0$.
The map $\triangledown$ makes sense on all of $\Spec(\Lambda_K)$ when we allow
poles in the image.

\begin{corollary} Let $\kappa$ be a $\mathcal{O}_K$ point of $\Lambda_K$.
The map $\pi$ is etale at the points above $\kappa$ if and only if for
all Dirichlet characters $\chi$ we have
$\triangledown(L_p(C, \chi)) \in \Gamma(\pi_*  
\mathcal{O}_C \otimes \Omega_{\Spec(\Lambda_K)}, V)$
where $V$ is some Zariski open containing $\kappa$.  
\end{corollary}

\begin{proof}
The "if" direction follows from the existence of the Gauss-Manin connection
for smooth maps.  For the "only if" let $x$ be a point in $\pi^{-1}(\kappa)$.  
Let $\chi$ be a Dirichlet character.  
By our hypothesis $\triangledown(L_p(C, \chi)) \in (\pi_*  
\mathcal{O}_C \otimes \Omega_{\Spec(\Lambda_K)})_\kappa$
Note that
\[ (\pi_* \mathcal{O}_C \otimes \Omega_{\Spec(\Lambda_K)})_{\kappa}
\cong A_{\kappa} \otimes_{\Lambda_{K,\kappa}} \Omega_{\Spec(\Lambda_K), \kappa}. \]
Choose a parameter $T_0$ of the weight space and let $D$ be the map from 
$\Omega_{\Spec(\Lambda_K)})_{\kappa} \to \mathcal{O}_{\Spec(\Lambda_K)}$
that sends $dT_0$ to $1$.  Then $D \circ \triangledown$ is the map
$A_{\kappa} \to A_{\kappa}$ given by differentiation
with respect to $T_0$.  There is a natural map $l: A_{\kappa} \to A_{(x)}$,
the localization of $A$ at $x$.  Since $\triangledown(L_p(C,\chi))$
is contained in $(\pi_* \mathcal{O}_C \otimes \Omega_{\Spec(\Lambda_K)})_{\kappa}$
we see that $l \circ D \circ \triangledown(L_p(C,\chi))$ 
is contained in $A_{(x)}$.
This means that $\frac{d}{dT_0}L_p(C,\chi)$ does not have a pole at $x$.
The corollary then follows from Theorem \ref{ramification theorem}

\end{proof}

\bibliographystyle{plain}
\bibliography{BibTex1}

\end{document}